\documentclass[12pt]{amsart}
\usepackage[pagewise]{lineno}
\usepackage{amsfonts,latexsym,rawfonts,amsmath,amssymb,amsthm,mathrsfs}
\usepackage[plainpages=false]{hyperref}
\usepackage{graphicx}
\usepackage{comment}
\usepackage[pagewise]{lineno}
\usepackage{xcolor}

\numberwithin{equation}{section}
\renewcommand{\theequation}{\arabic{section}.\arabic{equation}}
\RequirePackage{color}
\theoremstyle{plain}
\newtheorem{theorem}{Theorem}[section]
\newtheorem{lem}[theorem]{Lemma}
\theoremstyle{Corollary}
\newtheorem{cor}[theorem]{Corollary}

\newtheorem{remark}{Remark}[section]
\newcommand{\beq}{\begin{equation}}
	\newcommand{\eeq}{\end{equation}}
\newcommand{\beqs}{\begin{eqnarray*}}
	\newcommand{\eeqs}{\end{eqnarray*}}
\newcommand{\beqn}{\begin{eqnarray}}
	\newcommand{\eeqn}{\end{eqnarray}}
\newcommand{\beqa}{\begin{array}}
	\newcommand{\eeqa}{\end{array}}

\def\S{\mathbb S}
\def\R{\mathbb R}
\def\tR{\mathcal R}
\def\E{\mathscr E}
\def\C{\mathscr C}
\def\L{\mathscr L}
\def\B{\mathscr B}
\def\A{\mathscr A}
\def\tC{\mathcal C}
\def\M{\mathcal M}
\def\N{\mathcal N}
\def\G{\mathcal G}
\def\J{\mathcal J}
\def\I{\mathcal I}
\def\K{\mathcal K}
\def\O{\mathcal O}
\def\tS{\mathcal S}
\def\p{\partial}
\def\D{\nabla}
\def\bD{\overline \nabla}
\def\vol{\text{Vol}}
\def\area{\text{Area}}
\def\diam{\text{diam}}
\def\dist{\text{dist}}
\def\const{\text{const.}}
\def\vrho{\varrho}
\def\ve{\varepsilon}
\def\eps{\epsilon}
\def\graph{\text{graph}}
\def\h{\widetilde h}
\def\F{\widetilde F}
\def\lan{\langle}
\def\ran{\rangle}
\def\a{\widetilde a}
\def\vphi{\varphi}

\begin{document}

	\baselineskip16pt
	\parskip3pt
	
	\title{A flow approach to the Orlicz-Minkowski problem for torsional rigidity}
	
	\author{Weimin Sheng}
	\address{Weimin Sheng: School of Mathematical Sciences, Zhejiang University, Hangzhou 310058, China.}
	\email{weimins@zju.edu.cn}
	
	\author{Ke Xue}
	\address{Ke Xue: School of Mathematical Sciences, Zhejiang University, Hangzhou 310058, China.}
	\email{12235003@zju.edu.cn}
	
	\date{}
	
	\thanks{The authors were supported by NSFC, grant nos. 11971424 and 12031017.}
	
	\begin{abstract}
		
		In this paper the Orlicz-Minkowski problem for torsional rigidity, a generalization of the classical Minkowski problem, is studied. Using the flow method, we obtain a new existence result of solutions to this problem for general measures.
	\end{abstract}
	
	\maketitle

	\section{Introduction and overview of the main results} \label{introduction--1} \setcounter{equation}{0}
	
	During the past two decades it has been witnessed the great progress was achieved on the Minkowski type problems,
	and some new Minkowski type problems have been continuously emerged, for example, the
	$L_p$ Minkowski problem \cite{Luk93}, the Orlicz-Minkowski problem \cite{Luk10}, the dual Minkowski \cite{HLYZ16},
	the $L_p$ dual Minkowski \cite{LYZ18}, the dual Orlicz-Minkowski problems \cite{GHWXY18, GHXY18, SY20, ZXY18}, and the
	($L_p$ and Orlicz) Aleksandrov problem \cite{YH21, HLYZ18} etc.. There is a growing body researches addressed, we only name a few as follows \cite{BF19, BHP18, BLDZ13, BLYZZ17, BIS16, BIS21, CHZ18, CL19, CLZ17, CLZ19,  ChWang00, ChWang06, HZ18, HLYZ05, LSW16, LW13, LO95, Zhao17, Zhao18} etc.. More can be found in the references therein, see \cite{HLYZ16, LSYY21, SY20} for instance.
	These Minkowski type problems can be reformulated by Monge-Amp\`{e}re
	type equations (assuming smooth enough) and hence largely enrich the theory of the
	fully nonlinear PDEs. On the other hand, these Minkowski type problems greatly push the development of the Brunn-Minkowski theory of convex bodies forward. 
	For example, 
	the development of the Orlicz-Brunn-Minkowski theory owes greatly to the study of the Orlicz-Minkowski problem \cite{Luk10}.
	
	Furthermore, the Orlicz-Minkowski problem for other Borel measures, for example, the capacity and the torsional rigidity etc., with respect to the solution of a boundary-value problem
	has also been studied extensively, see \cite{CZ21, Colesanti15, Jerison96, KT21} etc. for the capacity and \cite{ChenDai20, Colesanti10, HuLiu21, HuLiu22, LiZhu20} etc.  for the torsional rigidity. The Orlicz-Minkowski problem for torsional
	rigidity was introduced by Li-Zhu \cite{LiZhu20}. Recall that the torsional rigidity of a convex body in $\R^n$ is described by Colesanti \cite{Colesanti05}
	\begin{equation}\label{def1}
		\frac{1}{T(\Omega)}=\inf\Big\{\frac{\int_{\Omega}|\overline{\nabla} U|^2dX}{(\int_{\Omega}|U|dX)^2}, U\in W_0^{1,2}(\Omega) : \int_{\Omega}|U|dX>0\Big\},
	\end{equation}
	where $\overline{\nabla}$ is the gradient in $\R^n$. It has been shown in \cite{Colesanti10} that, there exists a unique function $U$ such that
	
	\begin{equation}\label{TRV}
		T(\Omega)=\int_{\Omega}|\overline{\nabla} U|^2dX,
	\end{equation}
	where $U$ satisfies the boundary-value problem
	\begin{equation}\label{BVP}
		\left\{
		\begin{array}{ll}
			\Delta U(X)=-2, X\in \mathring{\Omega} ,\\\\
			U(X)=0,\quad   X\in \partial\Omega,
		\end{array}
		\right.
	\end{equation}
	for a convex body $\Omega$ we denote by $\mathring{\Omega}$ its interior and $\Delta$ the Lapalacian in $\mathbb{R}^n.$
	
	
	Colesanti and Fimiani in \cite{Colesanti10} built up the following variational formula of the torsional rigidity, for two arbitrary convex bodies $\Omega_0,\Omega_1$ in $\R^n$, 
	\begin{equation}\label{VFT}
		\frac{d}{dt}T(\Omega_0+t\Omega_1)|_{t=0^+}=\int_{\partial\Omega_0}h_{\Omega_1}(\nu_{\Omega_0}(X))|\overline{\nabla} U(X)|^2d\mathcal{H}^{n-1}(X)
	\end{equation}
	where $h_{\Omega_1}$ is the support function of $\Omega_1$, $\nu_{\Omega_0}$ is the Gauss mappign of $\partial\Omega_0$, $\mathcal{H}^{n-1}$ is the $(n-1)$-dimensional Hausdorff measure. By \eqref{VFT}, the torsional measure $\mu_{tor}(\Omega,\cdot)$ is defined on the unit sphere $\mathbb{S}^{n-1}$  by
	\begin{equation}\label{TRM}
		\mu_{tor}(\Omega,\eta)=\int_{\nu^{-1}_{\Omega}(\eta)}|\overline{\nabla} U(X)|^2d\mathcal{H}^{n-1}(X)=\int_{\eta}|\overline{\nabla} U(\nu^{-1}_{\Omega}(x))|^2dS_{\Omega}(x),
	\end{equation}
	for every Borel subset $\eta$ of $\mathbb{S}^{n-1}$. 
	Here $S_{\Omega}$ is the surface area measure of $\mathbb{S}^{n-1}$ and $\nu^{-1}_{\Omega}(x)$ is
	the inverse of the Gauss map on $\mathbb{S}^{n-1}$.
	Furthermore, Dahlberg \cite{Dahlberg77} revealed that $\overline{\nabla} U$
	has finite non-tangential limit
	$\mathcal{H}^{n-1}$ a.e. on $\partial\Omega$
	and $|\overline{\nabla} U|\in L^2(\partial\Omega,\mathcal{H}^{n-1}) $ without
	smoothness assumption on $\Omega$. Hence, \eqref{TRM} is well-defined on the unit sphere $\mathbb{S}^{n-1}$.
	
	Colesanti and Fimiani \cite{Colesanti10} first proposed the Minkowski problem for torsional rigidity: {\it{
			Given a finite Borel measure $\mu$ on the unit sphere $\mathbb{S}^{n-1}$,
			under what the necessary and sufficient conditions on $\mu$,
			does there exist a 
			convex body $\Omega$ of $\R^n$ such that $\mu=\mu_{tor}(\Omega,\cdot)$?}} They proved the existence and uniqueness up to translations of the solution
	via a variational argument which was first proposed by Aleksandrov \cite{Aleksandrov39, Aleksandrov38}.
	
	The $L_p$ version of Minkowski problem for torsional rigidity has
	been studied, see for instance, Chen-Dai \cite{ChenDai20} for $p > 1$ and Hu-Liu \cite{HuLiu21} for $0 < p < 1$.
	Recently, the $L_p$ Minkowski problem for torsional rigidity was extended to the
	Orlicz case involving non-homogeneous terms by Li-Zhu \cite{LiZhu20}, which asks: given a
	finite Borel measure $\mu$ on the unit sphere $\mathbb{S}^{n-1}$, what are the necessary and sufficient conditions on $\mu$ such that there exists a convex body $\Omega$ with support function $h$
	so that
	\begin{equation}\label{OMPT}
		hd\mu_{tor}(\Omega,\cdot) = \tau \,\psi(h)d\mu.
	\end{equation} 
	where $\psi\in\mathcal{A}$, the class of continuous functions $\psi: (0,\infty)\to (0,\infty)$ such that 
	\begin{itemize}
		\item[(i)]  $\Psi(s)=\int_0^s\frac{\psi(t)}{t}dt$ exists for all $s>0$;
		\item[(ii)]  $\lim_{s\rightarrow\infty}\Psi(s)=\infty$.
	\end{itemize}
	Hu, Liu and Ma \cite{HuLiu22} studied the existence of the Orlicz-Minkowski problem \eqref{OMPT} for even data. They proved the following:
	\begin{theorem}\label{HL}
		\cite{HuLiu22}	Suppose $\psi\in\mathcal{A}$. If $\mu$ is an even finite Borel measure on $\mathbb{S}^{n-1}$ which is not concentrated on any closed hemisphere of $\mathbb{S}^{n-1}$, then there exists an origin symmetric convex body $\Omega\in\mathcal{K}_e$
		such that \eqref{OMPT} holds.
	\end{theorem}
	Here $\mathcal{K}_e$ denotes the set of origin-symmetric convex bodies in $\mathbb{R}^n.$
	Li and Zhu \cite{LiZhu20} studied the general (not necessarily even) Orlicz-Minkowski problem and obtained the following:
	\begin{theorem}\label{LiZhu}
		\cite{LiZhu20}	Suppose $\psi\in\mathcal{A}$, and $s/\psi(s)$ tends to $+\infty$ as $s\rightarrow0^+$. If $\mu$ is a finite Borel measure on $\mathbb{S}^{n-1}$ which is not concentrated on any closed hemisphere of $\mathbb{S}^{n-1}$, then there exists a convex body $\Omega\in\mathcal{K}_0$
		such that \eqref{OMPT} holds.
	\end{theorem}
	$\mathcal{K}_0$ denotes the set of convex bodies in $\mathbb{R}^n$ with the origin in their interior.
	When $\psi(s)=s^p$, we note that Theorem \ref{HL} includes the even $L_p$ torsional-Minkowski problem for $p >0$, and Theorem \ref{LiZhu} includes the general $L_p$ torsional-Minkowski problem for $p >1$. There is no result about  the general Orlicz-Minkowski problem which can include the general $L_p$ torsional-Minkowski problem for $0 <p <1$. Inspired by \cite{LSYY21}, in this paper, we shall fill this gap. In \cite{HuLiu22, LiZhu20}, the authors considered the convex body $\Omega\in\mathcal{K}_0$, 
	but the weak solution to \eqref{OMPT} may vanish somewhere, so the
	associated convex body may contain the origin on its boundary.
	Hence our primary goal is to extend the Orlicz-Minkowski problem to $\Omega\in\mathcal{K}$. We consider to study the extended set of $\psi$ in the following ways. 
	Let $\mathcal{B}$ be the class of continuous functions $\psi: [0,\infty)\to [0,\infty)$ such that 
	\begin{itemize}
		\item[(i)]  $\psi(s)\in \mathcal{A}$; 
		\item[(ii)]  $\psi(0)=\lim_{s\to 0^+}\psi(s)=0$.
	\end{itemize}
	
	We get the following theorem: 
	\begin{theorem}\label{main1}
		Suppose $\psi\in\mathcal{B}$. If $\mu$ is a finite Borel measure on $\mathbb{S}^{n-1}$ which is not concentrated on any closed hemisphere of $\mathbb{S}^{n-1}$, then there exists a convex body $\Omega\in\mathcal{K}$ such that \eqref{OMPT} holds.
	\end{theorem}
	We would like to mention that Theorem \ref{main1} holds for $\psi=s^p$ with $p>0$. The proof of Theorem \ref{main1} is based on the study of a suitably designed parabolic flow  and the use of approximation argument. The idea of using the parabolic flow comes from the fact that Problem \eqref{OMPT} can be rewritten as a Monge-Amp\`ere type equation on $\mathbb{S}^{n-1}$. Assume the pregiven measure $\mu$ has a density $f$ with respect to $\,dx$, \eqref{OMPT}  reduces to solve the following Monge-Amp\`ere type equation on $\mathbb{S}^{n-1}$:
	\begin{eqnarray}\label{elliptic}\ \ \ \ \ \ \ \ 
		h|\overline{\nabla} U(hx+\nabla h)|^2\det(\nabla^2 h+hI)=\gamma f(x){\psi(h)},
	\end{eqnarray}
	where $\nabla$ and $\nabla^2$ are the gradient and Hessian operators with respect to an orthonormal frame on $\mathbb{S}^{n-1}$, $\gamma>0$ is a constant, $I$ is the identity matrix.
	
	Let  $f:\mathbb{S}^{n-1}\rightarrow (0,\infty)$ be  a smooth positive function, and $\Omega_0\in \mathcal{K}_0$ be  a convex body such that  $\mathcal{M}_0=\partial\Omega_0$ is a smooth and uniformly convex hypersurface. We consider the following  curvature flow,
	\begin{equation}\label{SSF}
		\left\{
		\begin{array}{ll}
			\frac{\partial{X}}{\partial{t}} (x,t)&=\left(-f(\nu){\psi} (h)|\overline{\nabla} U(hx+\nabla h,t)|^{-2}K+\eta(t)h\right)\nu ,\\\\
			X(x,0)&=X_0(x),
		\end{array}
		\right.
	\end{equation}
	where $X(\cdot,t): \mathbb{S}^{n-1}\to\mathbb{R}^{n}$ is  the embedding that parameterizes a family of convex hypersurfaces $\mathcal{M}_t$ (in particular, $X_0$ is the parametrization of $\mathcal{M}_0$), 
	$\Omega_t$ is the convex body inclosed by $\mathcal{M}_t$, $U(\cdot,t)$ is the solution of \eqref{BVP} in $\Omega_t$, $K$ denotes the Gauss curvature of $ \Omega_t$ at $X(x,t)$, $\nu$ denotes the unit outer normal of $ \Omega_t$ at $X(x,t)$, $h$ is the support function of  $ \Omega_t$, 
	and
	\begin{eqnarray}\label{eta def}
		\eta(t)=\frac{\int_{\mathbb{S}^{n-1}}f {\psi}(h) dx}{(n+2)T(\Omega_t)}.
	\end{eqnarray} 
	
	Assume the functions $f$ and $\psi$ are smooth enough. We show that,  under the condition that  
	\begin{equation}\label{good-condition-1}
			\liminf_{s\to0^+}\frac{s^n}{\psi(s)}=\infty,
		\end{equation}   
		the flow \eqref{SSF} deforms a smooth and uniformly convex hypersurface to a limit hypersurface satisfying \eqref{elliptic}. The key ingredient for such convergence is to establish a priori estimates; the $C^0$ estimate in this case can be  obtained by the maximum principle. However, when the condition  \eqref{good-condition-1} does not hold, i.e.
		\begin{equation}\label{good-condition-2}  
			\liminf_{s\to 0^+}\frac{s^n}{\psi(s)}< \infty, 
		\end{equation}   
		the $C^0$ estimate cannot be obtained by the maximum principle directly. In order to overcome this obstruction, we study a more careful designed flow replacing \eqref{SSF} with the function $\psi\in\mathcal{B}$ replaced by a smooth function $\widehat{\psi}_\varepsilon: [0,\infty)\rightarrow [0, \infty)$ defined as follows:
		\begin{equation}\label{hatpsi eps}
			\widehat{\psi}_{\varepsilon}(s)=\left\{
			\begin{array}{ll}
				\psi(s), &\textrm{if $s\ge2\varepsilon$},\\\\
				s^{n+\varepsilon},&\textrm{if $0\le s\le\varepsilon$},
			\end{array}
			\right.
		\end{equation} and $\widehat{\psi}_\varepsilon(s)\leq C_0$ for  $s\in(\varepsilon, 2\varepsilon)$ is chosen  so that $\widehat{\psi}_{\varepsilon}$ is  smooth on $[0, \infty)$ and $\widehat{\psi}_\varepsilon(s)>0$ for all $s\in(0,\infty)$, where $\varepsilon\in(0,1)$ is a small number. Hereafter, 
		\begin{eqnarray} \label{constant-c-0} 
				C_0=\max\{1, \max_{s\in[0,2] }\psi(s)\}.
		\end{eqnarray}
		In fact, we study the following curvature flow: 
		\begin{equation}\label{SF2}
			\left\{
			\begin{array}{ll}
				\frac{\partial{X_{\varepsilon}}}{\partial{t}} (x,t)&=\left(-f(\nu)\widehat{\psi}_\varepsilon(h_{\varepsilon})|\overline{\nabla} U_\varepsilon(h_{\varepsilon}x+\nabla h_{\varepsilon},t)|^{-2} K+\eta_\varepsilon(t)h_{\varepsilon}\right)\nu ,\\\\
				X_{\varepsilon}(x,0)&=X_0(x),
			\end{array}
			\right.
		\end{equation} 
		where $X_\varepsilon(\cdot,t):  \mathbb{S}^{n-1}\to\mathbb{R}^{n}$ parameterizes convex hypersurface
		$\mathcal{M}_t^\varepsilon$, 
		$h_{\varepsilon}$ denotes the support function of the convex body $\Omega^{\varepsilon}_t$
		inclosed by $\mathcal{M}_t^\varepsilon$, $U_\varepsilon(\cdot,t)$ is the solution of \eqref{BVP} in $\Omega^{\varepsilon}_t$ and 
		\begin{eqnarray}\label{eta epi def}
			\eta_\varepsilon(t)=\frac{\int_{\S^{n-1}}f\widehat{\psi}_\varepsilon(h) dx}{(n+2)T(\Omega^{\varepsilon}_t)}.
		\end{eqnarray}
		
		We will show  that  $h_\varepsilon(\cdot,t)$ is uniformly bounded from above in Lemma \ref{upperbound}. A positive uniform lower bound estimate for  $h_\varepsilon(\cdot,t)$ will be proved in Lemma \ref{lowerbound}, and this argument relies on the construction for $\widehat{\psi}_\varepsilon$ in \eqref{hatpsi eps}. By the $C^0$ estimates we further obtain the higher order estimates and therefore show that flow \eqref{SF2} exists for all time. These, together with Lemma \ref{monotone1},  imply the existence of a sequence of times $t_i\to\infty$ such that $h_\varepsilon(\cdot,t_i)$ converges to a positive and uniformly convex function $h_{\varepsilon,\infty}\in C^\infty(\mathbb{S}^n)$ solving the equation below
		\begin{eqnarray}\label{elliptic eq}\ \ \ \ \ \ \ \ 
			h|\overline{\nabla} U_\varepsilon(hx+\nabla h)|^2\det(\nabla^2h+hI)=\gamma_\varepsilon f(x){\widehat{\psi}_\varepsilon(h)},
		\end{eqnarray}
		where $\gamma_\varepsilon>0$ is a constant and $U_\varepsilon(\cdot)$ is the solution of \eqref{BVP} in $\Omega_{\varepsilon,\infty}$. Furthermore, we prove that there is a sequence of $\varepsilon_i\to0$ such that $h_{\varepsilon_i,\infty}$ locally uniformly converges to a weak solution of \eqref{elliptic}. Throughout this paper, we say that $h\in C^2(\mathbb{S}^{n-1})$, the set of function on $\mathbb{S}^{n-1}$ with continuous second order derivative, is uniformly convex if the matrix $\nabla^2h+hI$ is positively definite. Thus, the following theorem can be obtained, which provides solutions to \eqref{elliptic} 
		when $d\mu=fd\sigma_{\mathbb{S}^{n-1}}$.

		\begin{theorem}\label{main3}  Let  $\psi \in \mathcal{B}$ be a smooth function. Suppose that $\,d\mu(x)=f(x)\,dx$ with smooth and positive function $f$ on $\mathbb{S}^{n-1}$.
			The following statements hold:\\
			$(\mathrm{i})$ If $\psi$ satisfies \eqref{good-condition-1}, then one can find an $\Omega\in\mathcal{K}_0$  such that \eqref{elliptic};\\
			$(\mathrm{ii})$ If $\psi$ does not satisfy \eqref{good-condition-1}, i.e. it satisfies  \eqref{good-condition-2},  then one can find an  $\Omega\in\mathcal{K}$ such that \eqref{OMPT}.
		\end{theorem}
		
		For any general measure $\mu$,  there is a sequence of measures $\{\mu_i\}_{i\in \mathbb{N}}$,
		where $\,d\mu_i =f_i d\sigma_{\mathbb{S}^{n-1}}$ with $f_i$ being smooth and strictly positive on $\mathbb{S}^{n-1}$,
		such that $\mu_i$ converges to $\mu$ weakly.
		Theorem \ref{main1} is then proved by the virtue of Theorem \ref{main3} and an approximation argument.
		
		Next, we consider the Orlicz-Minkowski problem for torsional rigidity that includes the $L_p$ torsional Minkowski problem for $p=0$, which is called the logarithmic torsional Minkowski problem.
		Let $\mathcal{C}$ be the class of continuous functions $\psi: (0,\infty)\to (0,\infty)$ such that 
		\begin{itemize}
			\item[(i)]  $\Psi(s)=\int_1^s\frac{\psi(t)}{t}dt$ exists for all $s>0$;
			\item[(ii)]  $\lim_{s\rightarrow\infty}\Psi(s)=\infty$;
			\item[(iii)] For some $\widehat{C_1}>0$, we have $\int_{ \mathbb{S}^{n-1}}\Psi(|u\cdot x|)f(x)dx\geq-\widehat{C_1}, \quad  \forall u\in\mathbb{S}^{n-1}. $
		\end{itemize}
		We can get the following theorem:
		\begin{theorem}\label{main4}  Let  $\psi \in \mathcal{C}$ be a smooth function. Suppose that $\,d\mu(x)=f(x)\,dx$ with $f$ being even, smooth and strictly positive on $\mathbb{S}^{n-1}$, then there is a uniformly convex, even, smooth and positive solution to \eqref{elliptic}.
		\end{theorem}
		
		When $\psi(s)=1$, the Orlicz-Minkowski problem for torsional rigidity becomes the logarithmic torsional Minkowski problem, we claim that $\psi(s)=1\in \mathcal{C}$. Only need to show (iii) is true, to see this, we know $\int_{ \mathbb{S}^{n-1}}\log|x_i|du>-\infty$, where $(x_1,\cdot\cdot\cdot,x_n)$ denotes the Euclidean coordinates.
		The proof of the Theorem \ref{main4} relies on the flow \eqref{SSF}. Eveness and $\psi \in \mathcal{C}$ guarantee the $C^0$ estimate.

		This paper is organized as follows. In Section \ref{section-2},
		we recall some properties of convex hypersurfaces and torsional measure, and  present some properties of the flows \eqref{SSF} and \eqref{SF2}.
		We also show the monotonicity of functionals \eqref{function} and \eqref{functional} ( Lemmas \ref{mono} and \ref{monotone1}, respectively) and the preservation of $T(\cdot)$ along the flows ( Lemmas \ref{Volume} and \ref{monotone1}). 
		The $C^0$ estimates for the flows \eqref{SSF} and \eqref{SF2} when smooth functions $\psi$ satisfy the assumptions of Theorems \ref{main3}-\ref{main4} is established in Section \ref{section-3}.  
		Section \ref{section-4} dedicates to the proof of the long time existence of the flows \eqref{SSF} and \eqref{SF2} and the proof of Theorems \ref{main3}-\ref{main4}.  Moreover, the existence of solutions to Problem \ref{OMPT},  i.e., Theorem \ref{main1},  is proved by an approximation argument for general $\mu$ that is not concentrated on any closed hemisphere.
		Section \ref{section-5} collects the second derivative estimates which is used in the study of flows \eqref{SSF} and \eqref{SF2}.
		
		\section{Preliminary and properties of the flows}\label{section-2}

		\subsection{Convex Bodies}
		\label{Convex}
		{~}
		
		We say $\Omega\subset \mathbb{R}^{n}$ is a convex body if it is a compact convex set with nonempty interior. 
		Denote by $\mathcal{K}$ the set of all convex bodies in $\mathbb{R}^{n}$ containing the origin.
		Let $\mathcal{K}_0\subset \mathcal{K}$ be the set of all convex bodies with the origin in their interiors, $\mathcal{K}_e$ is the class of origin-symmetric convex bodies. For $\Omega\in\mathcal{K}$,  define its radial function $r_\Omega: \mathbb{S}^{n-1}\rightarrow [0, \infty)$ and support function $h_\Omega: \mathbb{S}^{n-1}\rightarrow [0, \infty)$, respectively, by  
		\begin{eqnarray}  \label{def u} \ \ \ \ 
			r_\Omega(x)=\max\{a \in \mathbb{R}: a x \in \Omega\} \ \ \mathrm{and} \ \ h_\Omega(x)=\max\{x\cdot y, \,y\in \Omega\}, \, \ \,x\in\mathbb{S}^{n-1}, 
		\end{eqnarray} where $x\cdot y$ denotes the inner product in $\mathbb{R}^{n}$. 
		
		For $\Omega\in\mathcal{K}$, let $\partial \Omega$ be its boundary. The Gauss map  of $\partial \Omega$, denoted by $\nu_{\Omega}: \partial \Omega\to\mathbb{S}^{n-1}$, is defined as follows: for $y\in \partial \Omega$, 
		\begin{eqnarray*}
			\nu_{\Omega}(y)=\{x\in\mathbb{S}^{n-1}: x\cdot y=h_\Omega(x)\}.
		\end{eqnarray*} Let $\nu_{\Omega}^{-1}: \mathbb{S}^{n-1}\rightarrow \partial \Omega$ be the reverse Gauss map such that 
		\begin{eqnarray*}
			\nu_{\Omega}^{-1}(x)=\{y\in\partial \Omega: x\cdot y=h_\Omega(x)\}, \ \ \ x\in \mathbb{S}^{n-1}.
		\end{eqnarray*}  Denote by $\alpha_\Omega: \mathbb{S}^{n-1}\rightarrow \mathbb{S}^{n-1}$   the radial Gauss image of $\Omega$. That is,  $$\alpha_{\Omega}(\xi)=\{x\in \mathbb{S}^{n-1}:  x\in \nu_{\Omega}(r_{\Omega}(\xi)\xi)\}, \ \ \xi\in \mathbb{S}^{n-1}. $$ Define $\alpha^*_{\Omega}: \mathbb{S}^{n-1}\rightarrow \mathbb{S}^{n-1}$, the reverse radial Gauss image of $\Omega$ as follows: for any Borel set $E\subset\mathbb{S}^{n-1}$,  
		\begin{equation}\label{rev-rad-gauss} \alpha^*_\Omega(E)=\{\xi\in\mathbb{S}^{n-1}: r_\Omega(\xi)\xi\in\nu_\Omega^{-1}(E)\}.\end{equation} 
		We often omit the subscript $\Omega$ in $r_{\Omega}$, $h_{\Omega}$, $\nu_{\Omega}$, $\nu_{\Omega}^{-1}$, $\alpha_{\Omega}$,  and $\alpha_{\Omega}^*$ if no confusion occurs.

		Let us recall some basic notations. In $\mathbb{R}^{n}$, sometimes, we also use $\langle x, y\rangle =x\cdot y$ to denote the inner product of $x, y\in \mathbb{R}^n$. $|x|$ means the Euclidean norm of $x\in \mathbb{R}^n$. The unit sphere $\S^{n-1}$ will be assumed to have a smooth local orthonormal frame field  $\{e_1, \cdots, e_{n-1}\}$. By $\nabla$ and $\nabla^2$, the gradient and Hessian operators with respect to $\{e_1, \cdots, e_{n-1}\}$ on $\S^{n-1}$. The surface area of $\S^{n-1}$ is denoted by $|\S^{n-1}|$.  
		
		Let  $\Omega\in \mathcal{K}$ be a convex body containing the origin. Denote by $w_\Omega:\S^{n-1}\to\mathbb{R}$  the width function of $\Omega$ which can be formulated by  \[w_\Omega(x)=h_\Omega(x)+h_\Omega(-x).\] We shall need $w_\Omega^+$ and $w_\Omega^-$,  the maximal and the  minimal width of $\Omega$, respectively, which can be formulated by
		\begin{eqnarray}\label{mmwide}
			w_\Omega^+=\max_{x\in\S^{n-1}}\{h_\Omega(x)+h_\Omega(-x)\}\ \ \mathrm{and}\ \ 
			w_\Omega^-=\min_{x\in\S^{n-1}}\{h_\Omega(x)+h_\Omega(-x)\}.
		\end{eqnarray} 
		The following result can be found in \cite{CL19} for instance.
		\begin{lem}\label{pro}
			Let $\Omega\in\mathcal{K}_0$ be a convex body containing the origin in its interior, $h_{\Omega}$ and $r_{\Omega}$ be the support and radial functions of $\Omega$, and $x_{max}\in \S^{n-1}$ and $\xi_{min}\in \S^{n-1}$ be such that $h_{\Omega}(x_{max})=\max_{x\in\S^{n-1}}h_{\Omega}(x)$ and $r_{\Omega}(\xi_{min})=\min_{\xi\in\S^{n-1}}r_{\Omega}(\xi)$. Then
			\begin{eqnarray*}
				\max_{x\in \S^{n-1}}h_{\Omega}(x)&=&\max_{\xi\in \S^{n-1}}r_{\Omega}(\xi)\ \ \mathrm{and}\ \ \min_{x\in\S^{n-1}}h_{\Omega}(x)=\min_{\xi\in \S^{n-1}}r_{\Omega}(\xi),\\
				h_{\Omega}(x)&\ge& \langle x,  x_{max}\rangle h_{\Omega}(x_{max}) \ \ \ \mathrm{for\ all}\ x\in\S^{n-1},\\
				r_{\Omega}(\xi) \langle\xi, \xi_{min}\rangle &\le&  r_{\Omega}(\xi_{min})\ \ \ \mathrm{for\ all}\ \xi\in\S^{n-1}.
			\end{eqnarray*}
		\end{lem}

		\subsection{Torsional rigidity and torsional measure}
		\label{TRTM}
		{~}
		
		According to \cite{Colesanti10,LiZhu20}, we first list some properties of torsional rigidity associated with the solution $U$ of \eqref{BVP} as follows:
		\eqref{TRV} can be transformed as
		
		\begin{align*}
			T(\Omega)&=\frac{1}{n+2}\int_{\partial\Omega}X\cdot\nu_{\Omega}(X)|\overline{\nabla} U(X)|^{2} d\mathcal{H}^{n-1}(X)\nonumber\\
			&=\frac{1}{n+2}\int_{\S^{n-1}}h_{\Omega}(x)|\overline{\nabla} U(\nu^{-1}_{\Omega}(x))|^{2} dS_{\Omega}(x).
		\end{align*}
		The torsional rigidity is positively homogeneous of order $(n +2)$, that is,
		\[T(t\Omega)=t^{n+2}T(\Omega) \quad \forall \Omega \in \K, t>0.\]
		This can be obtained from the fact that if $U$ is the solution of \eqref{BVP} in $\Omega$ and $t>0$, then the function
		\[V(y)=t^2U(\frac{y}{t}),\quad y\in t\Omega,\]
		is the corresponding solution in $t\Omega$.
		
		The torsional rigidity is also translation invariant ,that is, $T(x+\Omega)=T(\Omega), \forall x\in\R^n$. By the fact that $W(y)=U(y-x)$ is the corresponding solution in $x+\Omega$.\\
		\indent Obviously, by the definition \eqref{def1}, the torsional rigidity is monotone increasing, i.e.$T(\Omega_1)\leq T(\Omega_0)$ if $\Omega_1\subset\Omega_0.$
		
		The following lemma leads to an $L^\infty$ estimate for the gradient of $U$.
		\begin{lem}\label{Graes}\cite{Colesanti10}
			Let $\Omega$ be a convex body in 
			$\R^n$, if $U$ is the solution of \eqref{BVP} 
			in $\Omega$, then
			\[|\overline{\nabla} U(X)|\leq\diam(\Omega), \quad \forall X\in\Omega.\]
			
		\end{lem}
		For any $X\in\partial\Omega, 0<b<1,$ the non-tangential cone is defined as
		\[\Gamma(X)=\{Y\in \Omega: \dist(Y,\partial\Omega)>b|X-Y|\}.\]
		\begin{lem}\label{nonlimit}
			Let $\Omega$ be a convex body in 
			$\R^n$, if $U$ is the solution of \eqref{BVP}
			in $\Omega$, then the non-tangential limit 
			\[\overline{\nabla} U(X)=\lim_{Y\rightarrow X,Y\in\Gamma(X)}\overline{\nabla} U(Y),\]
			exists for $\mathcal{H}^{n-1}$ almost all $X\in\partial\Omega$. Furthermore, for $\mathcal{H}^{n-1}$ almost all $X\in\partial\Omega$,
			\[\overline{\nabla} U(X)=-|\overline{\nabla} U(X)|\nu_{\Omega}(X).\]
		\end{lem}
		
		The following two lemmas are imporant for us to solve the Minkowski problem for torisional rigidity(see \cite{Colesanti10,LiZhu20}):
		\begin{lem}\label{TVC}
			Let $\{\Omega_i\}_{i=0}^\infty$ be a sequence of convex bodies in $\R^n$, 
			if $\Omega_i$ converges to 
			$\Omega_0$ in the Hausdorff metric as $i\rightarrow\infty$, then
			\[\lim_{i\rightarrow\infty}T(\Omega_i)=T(\Omega_0).\]
		\end{lem}

		\begin{lem}\label{TMC}
			Let $\{\Omega_i\}_{i=0}^\infty$ 
			be a sequence of convex bodies in $\R^n$, if $\Omega_i$ convergens to  
			$\Omega_0$ in the Hausdorff metric as $i\rightarrow\infty$, then the sequence of $\mu_{tor}(\Omega_i,\cdot)$ converges weakly in the sense of measures
			to  $\mu_{tor}(\Omega_0,\cdot)$ as $i\rightarrow\infty.$
		\end{lem}

		\subsection{Monotonicity of functionals}
		\label{MOF}
		{~}

		Let $\mathcal{M}$ be a smooth, closed, uniformly convex hypersurface in $\mathbb{R}^{n}$, enclosing the origin.  The parametrization of $\mathcal{M}$ is given by the inverse Gauss map $X:\S^{n-1}\to \mathcal{M}\subset \mathbb{R}^{n}$. It follows from \eqref{def u} and \eqref{rev-rad-gauss} that 
		\begin{eqnarray}\label{X}
			X(x)&=&r(\alpha^*(x))\alpha^*(x),\\ 
			\label{support}
			h(x)&=&\langle x,X(x)\rangle,
		\end{eqnarray} where  $h$ is the support function of (the convex body inclosed by) $\mathcal{M}$.
		It is well known that the Gauss curvature of $\mathcal{M}$ is \begin{equation}
			K=\frac{1}{\det(\nabla^2h+hI)},\label{curvature-formula-1}
		\end{equation} and  the principal curvature radii of $\mathcal{M}$  are the eigenvalues of the matrix 
		\begin{equation}\label{pcr}
			b_{ij}=\nabla_{ij}h+h\delta_{ij}.
		\end{equation}  Moreover, we have the following equalities, see e.g. \cite{LSW16},
		\begin{eqnarray}  \label{nablau} 
			r\cdot\xi&=&h\cdot x+\nabla h\ \ \mathrm{and} \ \ r=\sqrt {h^2+|\nabla h|^2},\\ \label{nablar}
			h&=&\frac{r^2}{\sqrt{r^2+|{\nabla}r|^2}}.
		\end{eqnarray}
		
		Let $h(\cdot, t)$ and $r(\cdot, t)$ be the support and radial functions of $\mathcal{M}_t$. Recall that (see e.g., \cite[Lemma 2.1]{CCL19} or \cite{HLYZ16, LSW16})
		\begin{eqnarray}\label{r--u}
			\frac{\partial_t r(\xi,t)}{r }=\frac{\partial_t h (x,t)}{h}.
		\end{eqnarray} 
		By \eqref{support} and \eqref{curvature-formula-1}, the flow equation \eqref{SSF} for  $\mathcal{M}_t$ can be reformulated by its support function $h(x,t)$ as follows: 
		\begin{equation}\label{uSSF}
			\left\{
			\begin{array}{rl}
				\partial_th(x,t)&=-f(x){\psi}(h)|\overline{\nabla} U(hx+\nabla h,t)|^{-2} K+\eta(t)h,\\\\
				h(\cdot,0)&=h_0.
			\end{array}\right.
		\end{equation} By use of \eqref{r--u}, the flow equation \eqref{SSF} for  $\mathcal{M}_t$ can also be reformulated by its  radial function $r(\xi,t)$ as follows: 
		\begin{equation}\label{rSSF}
			\left\{
			\begin{array}{rl}
				\partial_tr(\xi,t)&=-f(x){\psi}(h)h^{-1}r|\overline{\nabla} U(r\cdot\xi,t)|^{-2} K+\eta(t)r ,\\\\
				r(\cdot,0)&=r_0.
			\end{array}\right.
		\end{equation} Similarly, \eqref{SF2} can be reformulated by its support function $h(x,t)=h_{\varepsilon}(x,t)$ and radial function $r(\xi,t)=r_{\varepsilon} (\xi,t)$ rep.
		\begin{equation}\label{uSF2}
			\left\{
			\begin{array}{rl}
				\partial_th(x,t)&=-f(x)\widehat{\psi}_\varepsilon(h)|\overline{\nabla} U(hx+\nabla h,t)|^{-2} K+\eta_\varepsilon(t)h,\\\\
				h(\cdot,0)&=h_0;
			\end{array}\right.
		\end{equation} 
		\begin{equation}\label{rSF2}
			\left\{
			\begin{array}{rl}
				\partial_tr(\xi,t)&=-f(x) \widehat{\psi}_\varepsilon(h) h^{-1} r|\overline{\nabla} U(r\cdot\xi,t)|^{-2}K+\eta_\varepsilon(t)r ,\\\\
				r(\cdot,0)&=r_0.
			\end{array}\right.
		\end{equation}
		
		It is clear that both of the equations \eqref{uSSF} and \eqref{uSF2} are parabolic Monge-Amp\'ere type, their solutions exist for a short time. Therefore the flow \eqref{SSF}, as well as \eqref{SF2}, has short time solution. 
		Let $\mathcal{M}_t=X(\mathbb{S}^{n-1}, t)$ be the smooth, closed and uniformly convex hypersurface parametrized by $X(\cdot, t)$, where $X(\cdot, t)$ is a smooth solution to the flow \eqref{SSF} with $t\in[0,T)$ for some constant $T>0$.  Let $\Omega_t$ be the convex body enclosed by $\mathcal{M}_t$ such that  $\Omega_t\in \mathcal{K}_0$ for all $t\in[0,T)$.  We now show that  $T(\Omega_t)$ remains unchanged along the flow \eqref{SSF}.

		\begin{lem}\label{Volume}  
			Let $X(\cdot,t)$ be a smooth solution to the flow \eqref{SSF} with $t\in[0,T)$, and $\M_t=X(\S^{n-1},t)$ be a smooth, closed and uniformly convex hypersurface.  Suppose that the origin lies in the interior of the convex body $\Omega_t$ enclosed by $\M_t$ for all $t\in[0,T)$. Then, for any $t\in[0,T)$, one has 
			\begin{eqnarray}\label{Volume bound}
				T(\Omega_t)=T(\Omega_0).
			\end{eqnarray}
		\end{lem}
		\begin{proof} 
			By \eqref{VFT}, \eqref{eta def} and \eqref{uSSF}, we have 
			\begin{align*}
				\frac{d}{dt}T(\Omega_t)
				&=\int_{\S^{n-1}}h_t|\overline{\nabla} U(hx+\nabla h,t)|^{2} K^{-1}dx\nonumber\\
				&=\int_{\S^{n-1}}-f{\psi}(h)dx+\eta(t)\int_{\S^{n-1}}h|\overline{\nabla} U(hx+\nabla h,t)|^{2} K^{-1}dx\nonumber\\
				&=0.
			\end{align*}
			In conclusion, $T(\Omega_t)$ remains unchanged along the flow \eqref{SSF}, and in particular, \eqref{Volume bound} holds for any $t\in[0,T)$.  
		\end{proof}

		The lemma below shows that the functional \begin{equation}\label{function}
			\mathcal{J}(h(\cdot,t))=\int_{\S^{n-1}}f{\Psi}(h)dx,
		\end{equation} is monotone along the flow \eqref{SSF}.
		\begin{lem}\label{mono} Let $\psi\in \mathcal{B}$ or $\psi\in \mathcal{C}$ be a smooth function. Let $X(\cdot,t)$, $\M_t$, and $\Omega_t$ be as in Lemma \ref{Volume}. Then the functional $\mathcal{J}$ defined by \eqref{function} is non-increasing along the flow \eqref{SSF}. That is, $\frac{\,d\mathcal{J}(h(\cdot, t))}{\,dt}\leq 0$, with equality if and only if $\mathcal{M}_t$ satisfies the elliptic equation \eqref{elliptic}.
		\end{lem}
		\begin{proof} Let $h(\cdot, t)$  be the support function of $\mathcal{M}_t$. By \eqref{eta def} and \eqref{uSSF}, we have
			\begin{align*}
				\frac{\,d\mathcal{J}(h(\cdot, t))}{\,dt}&= \int_{\mathbb{S}^{n-1}}f\frac{\psi(h)}{h}h_tdx\nonumber\\
				&=-\int_{\mathbb{S}^{n-1}}\frac{f^2\psi^2(h)K}{h|\overline{\nabla} U|^{2}}dx+\eta(t)\int_{\mathbb{S}^n}f(x)\psi(h)dx\nonumber\\
				&=\{(n+2)T(\Omega_t)\}^{-1}\{-\int_{\mathbb{S}^{n-1}}\frac{f^2\psi^2(h)K}{h|\overline{\nabla} U|^{2}}\int_{\S^{n-1}}h|\overline{\nabla} U|^{2} K^{-1}dx+(\int_{\mathbb{S}^n}f(x)\psi(h)dx)^2\}\nonumber\\
				&\leq 0,
			\end{align*}  
			where the first inequality holds if $\psi\in \mathcal{B}$, $\Psi(s)=\int_0^s\frac{\psi(t)}{t}dt$, or $\psi\in \mathcal{C}$, $\Psi(s)=\int_1^s\frac{\psi(t)}{t}dt$, the last inequality holds from H\"{o}lder inequality, and the equality holds if and only if there exists a constant $c(t)>0$ such that 
			\begin{eqnarray}\label{relsult1}
				h|\overline{\nabla} U(hx+\nabla h,t)|^2\det(\nabla^2 h+hI)=c(t) f(x){\psi(h)}.
			\end{eqnarray} Moreover, it can be proved that $c(t)=1/\eta(t)$ as follows: 
			\[\eta(t)=\frac{\int_{\mathbb{S}^{n-1}}f {\psi}(h) dx}{(n+2)T(\Omega_t)}=\frac{1}{c(t)}.\] This concludes the proof. 
		\end{proof}

		Regarding the flow \eqref{SF2}, the results similar to Lemmas \ref{Volume} and \ref{mono} can be obtained. For $\psi\in \mathcal{B}$, let \begin{eqnarray}\label{hat varphi}
			\widehat{\Psi}_\varepsilon(s)=\int_0^s\frac{\widehat{\psi}_\varepsilon(t)}{t}\,dt\ \ \ \mathrm{for}\ \  s\in [0,\infty).  
		\end{eqnarray}
		It follows from \eqref{hatpsi eps} that  $\widehat{\Psi}_\varepsilon: [0,\infty)\rightarrow [0, \infty)$ is well defined and  continuous, such that, $\widehat{\Psi}_\varepsilon(0)=0$. Moreover,  $\widehat{\Psi}_\varepsilon(s)$ is strictly increasing with respect to $s$ on $[0,\infty)$. Define the functional $\mathcal{J}_\varepsilon$ by
		\begin{equation}\label{functional}
			\mathcal{J}_\varepsilon(h)=\int_{\S^{n-1}}f\widehat{\Psi}_\varepsilon(h)dx.
		\end{equation}
		
		Following the same arguments of Lemmas \ref{mono} and \ref{Volume}, one can prove that, along the flow \eqref{SF2}, $\mathcal{J}_\varepsilon$  is non-increasing  and the torsional rigidity $T(\cdot)$ remains unchanged.  In fact, it is clear that, for $\varepsilon>0$ small enough, the solutions to the flow \eqref{SF2} exist for a short time.  Let $\mathcal{M}_t^\varepsilon=X_{\varepsilon} (\mathbb{S}^{n-1}, t)$ be the smooth, closed and uniformly convex hypersurface parametrized by $X_{\varepsilon}(\cdot, t)$, where $X_{\varepsilon}(\cdot, t)$ is a smooth solution to the flow \eqref{SF2} with $t\in[0,T)$ for some constant $T>0$.  Let $\Omega^{\varepsilon}_t$ be the convex body enclosed by $\mathcal{M}^{\varepsilon}_t$ such that  $\Omega_t^{\varepsilon}\in \mathcal{K}_0$ for all $t\in[0,T)$. $h_{\varepsilon}(x,t)$ is the support function of $\Omega_t^{\varepsilon}$.  
		\begin{lem}\label{monotone1}Let $\psi\in \mathcal{B}$ be a smooth function. Then, the functional $\mathcal{J}_\varepsilon$ is non-increasing along the flow \eqref{SF2}. That is, $\frac{\,d\mathcal{J_{\varepsilon}}(h_{\varepsilon}(\cdot, t))}{\,dt}\leq 0$, with equality if and only if $\mathcal{M}^{\varepsilon}_t$ satisfies the elliptic equation   \eqref{elliptic eq}. Moreover,  
			\begin{eqnarray}\label{G Volume bound}
				T(\Omega^{\varepsilon}_t)=T(\Omega_0),\,\forall\,t\in[0,T).
			\end{eqnarray}
		\end{lem}
		
		\section{$C^0$ estimates} \label{section-3}
		In this section, we shall estimate the uniform lower and upper bounds of the solutions to \eqref{SF2} and \eqref{SSF}. Let ${\Omega_0}\in \mathcal{K}_0$ be the convex body enclosed by the initial hypersurface $\mathcal{M}_0$ of the flow \eqref{SF2}. As $\varepsilon\in (0, 1)$  is arbitrarily small, a constant $c_0>0$ can be found, such that, 
		\begin{eqnarray}\label{ep}
			T(\Omega_0)\ge (2c_0)^{n+2}\ge(20\varepsilon)^{n+2}.
		\end{eqnarray}  For simplicity, we shall omit the subscript/superscript $\varepsilon$ in  $X_\varepsilon(\cdot,t)$  (the solution to \eqref{SF2}), $\mathcal{M}_t^\varepsilon=X_\varepsilon(\S^n,t)$, $\Omega_t^{\varepsilon}$ (the convex body enclosed by $\mathcal{M}_t^\varepsilon$),  $h_\varepsilon(\cdot,t)$ and $r_{\varepsilon}(\cdot, t)$  (the support and radial functions of $\Omega_t^{\varepsilon}$) etc, if  no confusion occurs.

		\begin{lem}\label{upperbound}
			Let $f$ is a smooth positive functions on $\mathbb{S}^{n-1}$ and $h(\cdot,t)$ be a positive, smooth and uniformly convex solution to \eqref{uSF2}.
			Let  $\psi\in \mathcal{B}$ be a smooth function.
			Then there is a constant $C>0$ depending only on $f$, $\psi$ and $\Omega_0$, but independent of $\varepsilon$ and $t$, such that
			\begin{eqnarray}\label{upper bound2}
				\max_{x\in\mathbb{S}^{n-1}}h(x, t)  \le  C \ \ \mathrm{and}\ \ \max_{x\in\mathbb{S}^{n-1}}|\nabla h (x, t)| \le C \ \ \ \mathrm{for\ all}\  t\in[0,T). 
			\end{eqnarray} 
		\end{lem}
		\begin{proof}  
				Combining \eqref{functional} wtih Lemma \ref{monotone1},  there exists a constant $C_1>0$, independent of $\varepsilon$, such that, 
				\begin{eqnarray}\label{intbound1} C_1\geq 
					\mathcal{J}_\varepsilon \big(h_0\big)\ge\mathcal{J}_\varepsilon\big(h(\cdot, t)\big)=\int_{\mathbb{S}^{n-1}}f(x)\widehat{\Psi}_\varepsilon(h)dx.
				\end{eqnarray}
				Let $x_t\in\mathbb{S}^{n-1}$ satisfy $h(x_t, t) =\max_{x\in \mathbb{S}^{n-1}} h(x, t)$. Without loss of generality, assume that $h(x_t, t)>10$ holds for some $t\in [0, T)$, otherwise, we have done. Let $$\Sigma_{\beta}=\{x\in \mathbb{S}^{n-1}: \langle x, x_t\rangle \ge \beta\}.$$ As $\varepsilon\in (0, 1)$, it can be checked from  Lemma \ref{pro} that, \begin{equation}\label{comp-max-supp}
					h(x,t)\ge h(x_t,t)\langle x_t,x\rangle\ge\frac{1}{2}h(x_t,t)>5>2\varepsilon \ \ \ \mathrm{for\ all}\ x\in \Sigma_{1/2}.\end{equation}
				It is easy to check that $\widehat{\Psi}_\varepsilon(s)\ge \Psi(s)-\Psi(2\varepsilon)\ge\Psi(s)-\Psi(5)$ for $s\in[2\varepsilon,\infty)$. Therefore, 
				\begin{eqnarray*}
					\int_{\mathbb{S}^{n-1}}f(x)\widehat{\Psi}_\varepsilon(h(x, t))\,dx&\ge&\int_{x\in\Sigma_{1/2}}f(x)\widehat{\Psi}_\varepsilon(h(x, t))\,dx
					\\&\ge&\int_{x\in\Sigma_{1/2}} f(x) \widehat{\Psi}_\varepsilon \Big(\frac{h(x_t,t)}{2}\Big)\,dx
					\\&\ge&  \Big[\Psi\Big(\frac{h(x_t,t)}{2}\Big)-\Psi(5)\Big] \cdot \min_{x\in \mathbb{S}^{n-1}} f(x)\cdot \int_{\Sigma_{1/2}}\,dx.
				\end{eqnarray*} Note that, the set $\Sigma_{1/2}$ may depend on $t$, but $ \int_{\Sigma_{1/2}}\,dx$ is a constant independent of $t$.
				Since  $\psi\in \mathcal{B}$, that $\Psi(s)$ is strictly increasing on $s\in (0, \infty)$ and $\lim_{s\rightarrow +\infty}\Psi(s)=+\infty$, this implies that $h(x, t)$ is uniformly bounded above.
				Hence it follows from \eqref{nablau} and Lemma \ref{pro} that 
				\[  \max_{x\in \mathbb{S}^{n-1}}|\nabla h(x, t)| \leq   \max_{x\in \mathbb{S}^{n-1}} r(x,t) =  \max_{x\in\mathbb{S}^{n-1}} h(x,t)\le C.\]
				
			\end{proof}

			The following lemma provides uniform bound for $\eta_\varepsilon(t)$ defined in \eqref{eta epi def}.

			\begin{lem}\label{etaepsbound}
				Let $f$ and $\psi$ satisfy conditions stated in Lemma \ref{upperbound}. Let $h(\cdot,t)$ be a positive, smooth and uniformly convex solution to \eqref{uSF2}. 
				Then  \begin{eqnarray}\label{etaepsilon bonud}
					\frac{1}{C_2}\le\eta_\varepsilon(t)\le C_2\ \ \ \mathrm{for\ all}\  t\in[0,T),
				\end{eqnarray} where $C_2>0$ is a constant  depending only on $f$, $\psi$ and $\Omega_0$, but independent of $\varepsilon$ and $t$. \end{lem}
			
			\begin{proof} 
				By Lemmas \ref{Graes} and \ref{monotone1},
				\begin{eqnarray}\label{le1}T(\Omega_0) =T(\Omega_t)=\int_{\Omega_t}|\overline{\nabla}U|^2dX \leq 
					(\diam(\Omega_t))^2\text{Vol}(\Omega_t)
				\end{eqnarray}As $\varepsilon\in (0, 1)$ be arbitrarily small,
				by \eqref{ep} and \eqref{le1}, we can get
				\begin{eqnarray}\label{maxlow1}
					\max_{x\in \mathbb{S}^{n-1}}r(x,t)\ge c_0 \ge 10\varepsilon.
				\end{eqnarray}
				Again let $x_t\in\mathbb{S}^{n-1}$ satisfy that $h({x_t},t)=\displaystyle\max_{x\in \mathbb{S}^{n-1}}h(x,t)$. Similar to \eqref{comp-max-supp}, for $x\in \Sigma_{1/2}$,  one has,
				\begin{eqnarray} \label{comp-5-15} h(x,t)\ge\frac{1}{2}h({x_t},t)\geq\frac{c_0}{2} \ge5\varepsilon.\end{eqnarray} 
				From \eqref{hatpsi eps}, the following fact holds: for  $x\in \Sigma_{1/2}$, 
				\begin{eqnarray}\label{widetilde1}
					\widehat{\psi}_\varepsilon\big(h(x,t)\big)=\psi\big(h(x,t)\big).
				\end{eqnarray}
					By \eqref{upper bound2}, \eqref{comp-5-15} and  \eqref{widetilde1}, one gets  
					\begin{eqnarray*} \int_{\mathbb{S}^{n-1}}f\widehat{\psi}_\varepsilon(h(x,t)) dx\geq \!\! 
						\int_{\Sigma_{1/2}}\!\!\!\!\! f\psi(h(x,t))\,dx    \geq \!   \min_{h\in [\frac{c_0}{2}, C]}\psi\big(h\big) \cdot \min_{x\in \mathbb{S}^{n-1}} f(x)\cdot \int_{\Sigma_{1/2}}\!\! \!\! \!\! \,dx. \end{eqnarray*} 
					This, together with \eqref{eta epi def}, further imply that, for all $t\in (0, T]$, 
					\begin{eqnarray} 
						\eta_\varepsilon(t)&=&\frac{\int_{\S^{n-1}}f\widehat{\psi}_\varepsilon(h) dx}{(n+2)T(\Omega^{\varepsilon}_t)} \nonumber \\ &\geq& \frac{\min_{h\in [\frac{c_0}{2}, C]}\psi\big(h\big) \cdot \min_{x\in \mathbb{S}^{n-1}} f(x)\cdot \int_{\Sigma_{1/2}}\!\! \!\! \!\! \,dx}{(n+2)T(\Omega_0)}>0.  \label{first-ineq}
					\end{eqnarray}

						It follows from  \eqref{hatpsi eps} and \eqref{constant-c-0} that  $\widehat{\psi}_\varepsilon (s)\le C_0$ on $[0,2\varepsilon]$, where $ C_0=\max\{1, \displaystyle\max_{s\in[0,2]}\psi(s)\}$, and $\widehat{\psi}_\varepsilon(s)=\psi(s)$ on $[2\varepsilon,\infty)$.
						Hence, for all $t\in [0, T)$, 
						\begin{eqnarray*}  
							\int_{\mathbb{S}^{n-1}}f\widehat{\psi}_\varepsilon(h(x, t))\,dx&=& \int_{\{x\in\S^n: h(x)\in[0,2\varepsilon)\cup[2\varepsilon,C]\}} f\widehat{\psi}_\varepsilon(h(x, t))\,dx 
							\\&\le&\int_{\mathbb{S}^{n-1}} \Big(C_0+\psi\big(h(x,t)\big)\Big)f\,dx \\ &\le& \Big\{C_0+\max_{h\in [0,C]} \psi\big(h\big)\Big\}\cdot\max_{x\in\mathbb{S}^{n-1}} f(x)\cdot |\S^{n-1}|. 
						\end{eqnarray*} 
						One has, for all $t\in [0, T)$,
						\begin{eqnarray*} 
							\eta_\varepsilon(t)=\frac{\int_{\mathbb{S}^{n-1}}f\widehat{\psi}_\varepsilon(h(x,t)) dx}{(n+2)T(\Omega^{\varepsilon}_t)}\leq \frac{\Big\{C_0+\max_{h\in [0,C]} \psi\big(h\big)\Big\}\cdot\max_{x\in\mathbb{S}^{n-1}} f(x)\cdot |\mathbb{S}^{n-1}| 
							}{(n+2)T(\Omega_0)}. 
						\end{eqnarray*} In view of \eqref{first-ineq}, one can let $C_2>0$ be independent of $\varepsilon$ and $t\in [0, T)$ can be found so that 
						\eqref{etaepsilon bonud} holds.  
					\end{proof}
					
					Recall that the maximal and the  minimal width of $\Omega$ defined in \eqref{mmwide}, respectively, are
					\begin{eqnarray*} 
						w_\Omega^+=\max_{x\in\mathbb{S}^{n-1}}\{h_\Omega(x)+h_\Omega(-x)\}\ \ \mathrm{and}\ \ 
						w_\Omega^-=\min_{x\in\mathbb{S}^{n-1}}\{h_\Omega(x)+h_\Omega(-x)\}.
					\end{eqnarray*} 
					The following lemma gives upper and lower bounds, which are independent of $\varepsilon$, for the maximal and the minimal width of $\Omega_t$ on $t\in [0, T)$.
					
					\begin{lem}\label{widthbound}Let $f$ and $\psi$ satisfy conditions stated in Lemma \ref{upperbound}.   Let $h(\cdot,t)$ be a positive, smooth and uniformly convex solution to \eqref{uSF2}.  
						Then there is a constant $C_3>0$ depending only on $f$, $\psi$ and $\Omega_0$, but independent of $\varepsilon$ and $t\in [0, T)$, such that,  for all $t\in[0,T)$, 
						\begin{eqnarray}\label{width bonud}
							1/C_3 \le w_{\Omega_t}^-\le w_{\Omega_t}^+\le C_3.
						\end{eqnarray}
					\end{lem}
					\begin{proof} 
						By Lemma \ref{upperbound}, one sees that $$w_{\Omega_t}^+=\max_{x\in\mathbb{S}^{n-1}}\{h_\Omega(x)+h_\Omega(-x)\}\leq 2C.$$
						
						On the other hand, by Lemmas \ref{Graes} and \ref{monotone1},
						\begin{eqnarray*}
							T(\Omega_0)=T(\Omega_t)=\int_{\Omega_t}|\overline{\nabla} U|^2dX \leq 
							(\diam(\Omega_t))^2\text{Vol}(\Omega_t)\leq C_0 w_{\Omega_t}^-.
						\end{eqnarray*}
						So, the desired constant $C_3$ can be obtained such that \eqref{width bonud} holds.
					\end{proof}
					
					\begin{remark}
						Let $U(\cdot,t)$ is the solution of \eqref{BVP} in $\Omega_t$ such that
						\begin{equation*}
							\left\{
							\begin{array}{ll}
								\Delta U(X,t)=-2, X\in \mathring{\Omega}_t ,\\\\
								U(X,t)=0,\quad   X\in \partial\Omega_t
							\end{array}
							\right.
						\end{equation*}
						we can get some results for $\Omega_t$ and $U(\cdot,t)$ as follows :
						\begin{itemize}\label{remark}
							\item[(i)] By Lemmas \ref{upperbound} and  \ref{widthbound}, we can get that there exist two constants $0<R_0<R_1$ such that $B_{R_0}\subset \Omega_t\subset B_{R_1}$ for all $t\in[0,T)$, where $B_{R_i}$ is a ball with radius $R_i,i=0,1$.
							\item[(ii)] Since $\Omega_t$ is smooth convex body, by global regularity theorem in \cite{GT98} we can get $U(\cdot,t)\in C^{\infty}(\Omega_t)$; 
							by the strong maximum principle, $U(\cdot,t)>0$ in $\mathring{\Omega}_t$.
							\item[(iii)]By (i), we can get that $\forall t\in[0,T), \forall x \in \partial\Omega_t$, there exist a ball $B$ such that $B\bigcap\partial\Omega_t={x}$ and $\Omega_t\subset B$ with radius $R$ which can be chosen to be independent of $t$ and $x$. Let $\bar{U}$ be the solution of \eqref{BVP} in $B_R$, by comparison pricinple, there exist a positive constant $\bar{c}$, independent of $t$, such that $U(\cdot,t)\leq\bar{U}(\cdot)\leq\bar{c}$ in $\Omega_t$.
							Similarly, there exists a ball $B$ such that $B\bigcap\partial\Omega_t={x}$ and $B\subset \Omega_t $ with radius $r$ which can be chosen to be independent of $t$ and $x$. Let $\widehat{U}$ is the solution of \eqref{BVP} in $B_r$, by comparison pricinple, $U(\cdot,t)\geq\widehat{U}(\cdot)$ in $B_r$, and $U(x,t)=\widehat{U}(x)$, we have $|\overline{\nabla} U(x,t)|\geq|\overline{\nabla}\widehat{U}(x)|\geq\widehat{c}$ for some $\widehat{c}>0$, independent of $t$, where the last inequality hold by Hopf lemma.
							\item[(iv)] $\forall t\in[0,T), \forall x \in \partial\Omega_t$, by Lemmas \ref{Graes}  and \ref{upperbound}, then there exists a positive constant $c>0$,  independent of $t$, such that $|\overline{\nabla} U(x,t)|\leq c$. Furthermore $|\overline{\nabla}^kU(x,t)|\leq \widetilde{c}$ by virtue of Schauder's theory(see Lemma 6.4 in \cite{GT98}) for the positive integer $k\geq2$, where $\widetilde{c}$ is a positive constant independent of $t$.
							
						\end{itemize}
						
					\end{remark}

					The following lemma provides a uniform lower bound for $h(x, t)$ on $t\in [0, T)$. 
					\begin{lem}\label{lowerbound} Let $f$ and $\psi$ satisfy conditions stated in Lemma \ref{upperbound}.  Let $h(\cdot,t)$ be a positive, smooth and uniformly convex solution to \eqref{uSF2}.  
						Then there is a constant $C_{\varepsilon}>0$ depending only on $\varepsilon$, $f$, $\psi$ and $\Omega_0$ but independent of $t$, such that, for all $t\in[0,T)$,
						\begin{eqnarray}\label{lower bound epsilon}
							\min_{x\in \mathbb{S}^{n-1}}h(x, t)\ge1/C_\varepsilon.
						\end{eqnarray}
					\end{lem}

					\begin{proof} Let $\bar{x}_t\in \mathbb{S}^{n-1}$ be such that $h(\bar{x}_t,t)= \min_{x\in \mathbb{S}^{n-1}}h(x,t)$. Without loss of generality, assume that  $h(\bar{x}_t,t) <\varepsilon$ holds for some $t\in [0, T)$.   
						It follows from \eqref{hatpsi eps} and  the fact that $\nabla^2 h(\bar{x}_t,t)$ is positive semi-definite that 
						\begin{eqnarray} \widehat{\psi}_\varepsilon(h(\bar{x}_t,t))=\left(h(\bar{x}_t,t)\right)^{n+\varepsilon},\nonumber \\  
							\det(\nabla^2 h(\bar{x}_t,t)+h(\bar{x}_t,t)I)\geq h(\bar{x}_t,t)^{n-1}. \label{estimate-curvature}\end{eqnarray} Together with \eqref{curvature-formula-1},\eqref{uSF2} and Remark \ref{remark}, one has
						\begin{eqnarray*}
							\partial_th(\bar{x}_t,t)&\ge&-\frac{f(\bar{x}_t)}{\widehat{c}^{2}} \left(h(\bar{x}_t,t)\right)^{n+\varepsilon}\left(h(\bar{x}_t,t)\right)^{1-n} +h(\bar{x}_t,t)\eta_\varepsilon(t)
							\\&\ge&\frac{f(\bar{x}_t)}{\widehat{c}^{2}}h(\bar{x}_t,t)\left(-\left(h(\bar{x}_t,t)\right)^\varepsilon+\frac{\widehat{c}^{2}\eta_\varepsilon(t)}{f(\bar{x}_t)}\right).
						\end{eqnarray*}
						This implies either $\partial_th(\bar{x}_t,t)\ge0$, or by \eqref{etaepsilon bonud}
						\[h(\bar{x}_t,t)>\left(\frac{\widehat{c}^{2}\eta_\varepsilon(t)\displaystyle}{\displaystyle\max_{x\in\mathbb{S}^{n-1}}f(x)}\right)^\frac{1}{\varepsilon}>\left(\frac{\widehat{c}^{2}\displaystyle}{C_2\displaystyle\max_{x\in\mathbb{S}^{n-1}}f(x)}\right)^\frac{1}{\varepsilon}.\]
						One gets 
						\begin{eqnarray}\label{lower}
							\min_{x\in \mathbb{S}^{n-1}}h(x, t)\ge\min\bigg\{\varepsilon,  \left(\frac{\widehat{c}^{2}\displaystyle}{C_2\displaystyle\max_{x\in\mathbb{S}^{n-1}}f(x)}\right)^\frac{1}{\varepsilon},h(\bar{x}_0,0)\bigg\}=:\frac{1}{C_\varepsilon},
						\end{eqnarray} where $C_{\varepsilon}>0$ is a constant independent of $t\in [0, T)$. This completes the proof.
					\end{proof}
					
					Note that the uniform lower bound for  $h(\cdot,t)$ in Lemma \ref{lowerbound} depends on $\varepsilon\in (0,1)$. However, such a lower bound may go to zero as $\varepsilon\to0$.

					The following two lemmas give the $C^0$ estimates for the flow \eqref{SSF}.

					\begin{lem}  Let $f$ is a smooth positive functions on $\mathbb{S}^{n-1}$. Let $\psi\in \mathcal{B}$ be a  smooth function such that   \eqref{good-condition-1} holds.     Let $h(\cdot,t)$ for $t\in[0,T)$ be a positive, smooth and uniformly convex solution to \eqref{uSSF}.
						Then there is a constant $C>0$ depending only on $f$, $\psi$ and $\Omega_0$ but independent of $\varepsilon$ and $t\in [0, T)$, such that, for all $t\in [0, T)$, 
						\begin{eqnarray*} 
							1/C\le \min_{x\in \mathbb{S}^{n-1}} h(x,t)\le\max_{x\in \mathbb{S}^{n-1}} h(x,t)\le C \ \ \mathrm{and} \ \  
							\max_{\mathbb{S}^{n-1}}|\nabla h(x, t)|\le C.
						\end{eqnarray*}
					\end{lem}
					\begin{proof} Consider the flow defined in \eqref{SSF} or \eqref{uSSF}.  
						By the proofs in   Lemmas \ref{upperbound} and \ref{etaepsbound},  the following statements hold for $t\in [0, T)$: 
						\begin{eqnarray}
							\label{upper11}\max_{x\in \mathbb{S}^{n-1}}h(x,t)\le C,\\
							\label{gradient11}\max_{x\in \mathbb{S}^{n-1}}|\nabla h|(x,t)\le C,\\
							1/C\le\eta(t)\le C,\label{etabound1}
						\end{eqnarray}
						where the constant $C>0$ (abuse of notation) depends only on $f, \psi$ and  $\Omega_0$ but is independent of $t\in [0, T)$.  
						
						Let $\bar{x}_t\in\mathbb{S}^{n-1}$ satisfy $h(\bar{x}_t, t)= \min_{x\in \mathbb{S}^{n-1}}h(x,t)$. Clearly, $\nabla h(\bar{x}_t,t)=0$ and $\nabla^2 h(\bar{x}_t,t)\ge 0$. A calculation similar to that in Lemma \ref{lowerbound}, combining \eqref{estimate-curvature} and Remark \ref{remark}, shows that \begin{eqnarray}
							\nonumber\partial_t h(\bar{x}_t,t) &\geq&-\frac{f(\bar{x}_t)}{\widehat{c}^{2}}\psi(h(\bar{x}_t,t))\left(h(\bar{x}_t,t)\right)^{1-n} +h(\bar{x}_t,t)\eta(t) \\&=&\frac{\psi(h(\bar{x}_t,t))}{h^{n-1}(\bar{x}_t,t)}\eta(t)\left(\frac{h^{n}(\bar{x}_t,t)}{\psi(h(\bar{x}_t,t))}-\frac{f(\bar{x}_t)}{\widehat{c}^{2}\eta(t)}\right).\label{p11}
						\end{eqnarray}
						
						Recall that $\psi$ satisfy \eqref{good-condition-1}, i.e., 
							$\displaystyle\liminf_{s\to0^+}\frac{s^n}{\psi(s)}=\infty$.
						Let $M$ be a finite constant such that $$M>\max_{(x,t)\in  \mathbb{S}^{n-1}\times[0,T)}\Big(\frac{f(x)}{\widehat{c}^{2}\eta(t)}\Big).$$
						Then, there exists $\delta_0$ depending on $M$ so that for $s\in(0,\delta_0)$,
						\begin{equation}\label{M}
							\frac{s^n}{\psi(s)}\geq M.
						\end{equation}
						Without loss of generality, let $h(\bar{x}_t,t)<\delta_0$.  By \eqref{etabound1}, \eqref{p11} and \eqref{M}, one has 
						\begin{eqnarray*}
							\partial_t h(\bar{x}_t,t) \geq \frac{\psi(h(\bar{x}_t,t))}{h^{n-1}(\bar{x}_t,t)}\eta(t)\left(M-\frac{f(\bar{x}_t)}{\widehat{c}^{2}\eta(t)}\right)\geq0.
						\end{eqnarray*} This further yields, for all $t\in [0, T)$,  \begin{eqnarray}
							h(\bar{x}_t,t)=\min_{ \mathbb{S}^{n-1}}h(\cdot,t)\ge\min\big\{\delta_0, h(\bar{x}_0,0)\big\}=: 1/C. \label{lower11}
						\end{eqnarray} 
					\end{proof}

					\begin{lem}{\label{LogC0} } Let $f$ is an even, smooth positive functions on $\mathbb{S}^{n-1}$. Let $\psi\in \mathcal{C}$ be a smooth function. Let $h(\cdot,t)$ for $t\in[0,T)$ be a positive, even, smooth and uniformly convex solution to \eqref{uSSF}.
						Then there is a constant $C>0$ depending only on $f$, $\psi$ and $\Omega_0$, such that, for all $t\in [0, T)$, 
						\begin{eqnarray*} 
							1/C\le \min_{x\in \mathbb{S}^{n-1}} h(x,t)\le\max_{x\in \mathbb{S}^{n-1}} h(x,t)\le C \ \ \mathrm{and} \ \  
							\max_{\mathbb{S}^{n-1}}|\nabla h(x, t)|\le C.
						\end{eqnarray*}
					\end{lem}
					
					\begin{proof}
						Since $\psi\in \mathcal{C}$, that $\Psi(s)$ is non-decreasing, and
						
						\begin{equation*}
							\left\{
							\begin{array}{ll}
								\Psi(s)\geq 0,\quad  s\geq1 ,\\\\
								\Psi(s)\leq0,\quad   0<s\leq1.
							\end{array}
							\right.
						\end{equation*}

						Let $x_t\in\mathbb{S}^{n-1}$ satisfy  $h(x_t, t) =\max_{x\in \mathbb{S}^{n-1}} h(x, t)$. Without loss of generality, we assume that $h(x_t, t)>10$ holds for some $t\in [0, T)$. Otherwise, we have done. Let $$\Sigma_{\beta}=\{x\in \mathbb{S}^{n-1}: |\langle x, x_t\rangle |\ge \beta\}.$$
						
						Combining with \eqref{function} and Lemma \ref{mono},  there exists a constant $C_0>0$, such that, 
						\begin{eqnarray}\label{intbound1}
							C_0\geq 
							\mathcal{J} \big(h_0\big)&\ge&\mathcal{J}\big(h(\cdot, t)\big)=\int_{\mathbb{S}^{n-1}}f(x)\Psi(h)dx\nonumber\\
							&\geq&\int_{\mathbb{S}^{n-1}}f(x)\Psi(h(x_t,t)|x\cdot x_t|)dx\nonumber\\
							&\geq&\Psi(\frac{1}{2}h(x_t,t))\int_{\Sigma_{1/2}}f(x)dx+\int_{\mathbb{S}^{n-1}\setminus \Sigma_{1/2}}f(x)\Psi(|x\cdot x_t|)dx\nonumber\\
							&\geq&\Psi(\frac{1}{2}h(x_t,t))\int_{\Sigma_{1/2}}f(x)dx+\int_{\mathbb{S}^{n-1}}f(x)\Psi(|x\cdot x_t|)dx\nonumber\\
							&\geq&\Psi(\frac{1}{2}h(x_t,t))\int_{\Sigma_{1/2}}f(x)dx-\widehat{C_1}\nonumber.
						\end{eqnarray}
						where the last inequality is because of $\psi\in \mathcal{C}$. Since $\lim_{s\rightarrow +\infty}\Psi(s)=+\infty$, it implies that $h(x, t)$ is uniformly bounded above.
						
						From the proof of Lemma \ref{widthbound}, we can get $\omega_{\Omega_t}^{-}\geq 1/C_3$, since $\Omega_t$ is O-symmetric, so $h(x, t)$ has a uniformly lower bound. The $C^0$ estimate can implies $C^1$ estimate, the lower and upper bounds of $\eta(t)$ and Remark \ref{remark}.

					\end{proof}

					\section{Solutions to the extended Orlicz problem for  Torsional Rigidity} \label{section-4} 
					This section is devoted to prove the existence of solutions to the extended Orlicz-Minkowski problem for torsional rigidity. 
					We first prove Theorem \ref{main3} in Section \ref{sec-main-3-1}, 
					which provides a smooth solution to the extended Orlicz-Minkowski problem for torsional rigidity, and Theorem \ref{main4} can aslo be proved in this section. 
					By the technique of approximation, Theorem \ref{main1} is proved in Section \ref{sec-main-1-1}. 
					This gives a solution to the extended Orlicz-Minkowski problem for torsional rigidity for general data.

					\subsection{Proof of Theorems \ref{main3} and  \ref{main4}} \label{sec-main-3-1} 
					{~}
					
					With the help of the $C^2$-estimates established in Lemmas \ref{lowerboundsigma_n} and \ref{boundsprincipleradii} in Section \ref{section-5}, we can prove Theorem \ref{main3}.  
					
					\begin{proof} [Proof of Part $(\mathrm{ii})$ in Theorem \ref{main3}.] Following the notations in Section \ref{section-3}, for  $\Omega_0\in\mathcal{K}_0$, let 
						$\mathcal{M}_0=\partial \Omega_0$ be a closed, smooth and uniformly convex hypersuface. The support function of $\Omega_0$ is $h_0$ and assume \eqref{ep}, namely $T(\Omega_0)\ge (2c_0)^{n+2}\ge(20\varepsilon)^{n+2}$ for some constant $c_0>0$ and for small $\varepsilon \in (0, 1)$. Let $X_\varepsilon(\cdot,t)$ be the solution to \eqref{SF2}, $h_\varepsilon(\cdot,t)$ be the support function of $\Omega_t^\varepsilon$ and  $\mathcal{M}_t^\varepsilon=\partial \Omega_t^\varepsilon$. Let $T$ be the maximal time such that $h_\varepsilon(\cdot,t)$ is positive, smooth and uniformly convex for all $t\in[0,T)$.  
						
						First of all, under the conditions on $f$ and $\psi$ as stated in  Part $(\mathrm{ii})$ of Theorem \ref{main3}, one can find a constant $\overline{C}_\varepsilon>0$ (depending on $\varepsilon$ but independent of $t$), such that, the principal curvature radii of $\mathcal{M}^{\varepsilon}_t$ is bounded from above and below. That is, for all $t\in [0, T)$, one has, 
						\begin{eqnarray}\label{sigma lower bonud}
							\overline{C}_\varepsilon^{-1} I\le\nabla^2h_{\varepsilon}(\cdot, t)+h_{\varepsilon}(\cdot, t)I\le \overline{C}_\varepsilon I.
						\end{eqnarray} In fact, conditions \eqref{cond1}, \eqref{cond2} and \eqref{cond3} follow from Lemmas \ref{upperbound}-\ref{lowerbound}.   Hence,  \eqref{sigma lower bonud} is a direct consequence of Lemmas \ref{lowerboundsigma_n} and \ref{boundsprincipleradii}, where $h(\cdot, t)$, $\eta(t)$ in Lemmas \ref{lowerboundsigma_n} and \ref{boundsprincipleradii} are replaced by $h_{\varepsilon} =h_{\varepsilon}(\cdot, t)$,  $\eta_\varepsilon(t)$.

						In view of \eqref{sigma lower bonud} and Lemmas \ref{upperbound}-\ref{lowerbound}, the flow \eqref{uSF2} is uniformly parabolic and then $|\partial_th_\varepsilon|_{L^\infty( \mathbb{S}^{n-1}\times[0,T))}\le \overline{C}_\varepsilon$ (abuse of notation $\overline{C}_\varepsilon$). It follows from results of Krylov and Safonov\cite{KS81} that  H\"{o}lder continuity estimates for $\nabla^2h_{\varepsilon}$ and $\partial_th_{\varepsilon}$  can be obtained. Hence, the standard theory of linear uniformly parabolic equations can be used to get the higher derivatives estimates, which further implies the the long time existence of a positive, smooth and uniformly convex solution to the flow \eqref{uSF2}. Moreover, 
						\begin{eqnarray}
							\label{est1}&&1/C_\varepsilon\le h_\varepsilon(x,t)\le C\ \ \ \mathrm{for\ all}\ (x,t)\in \mathbb{S}^{n-1}\times[0,\infty),\\
							\label{est2}&&1/C\le|\eta_\varepsilon(t)|\le C\ \ \ \mathrm{for\ all}\ t\in[0,\infty),\\
							\label{est3}&&\nabla^2h_\varepsilon+h_\varepsilon I\ge1/C_\varepsilon\ \ \ \mathrm{on}\ \, \mathbb{S}^{n-1}\times[0,\infty),\\
							\label{est4}&&|h_\varepsilon|_{C_{x,t}^{k,l}( \mathbb{S}^{n-1}\times[0,\infty))}\le C_{\varepsilon,k,l},
						\end{eqnarray}
						where $C_\varepsilon$ and $ C_{\varepsilon,k,l}$ are constants depending on $\varepsilon,\psi, f$ and  $\Omega_0$. Together with Lemma \ref{monotone1},  one has $0\leq \mathcal{J}_\varepsilon(h_{\varepsilon}(\cdot,t)) \le \mathcal{J}_\varepsilon(h_{0})$ for all $t\in[0,\infty)$, and hence   \begin{eqnarray}\label{J equ}  \int_{0}^{\infty} \Big|\frac{d}{dt}\mathcal{J}_\varepsilon(h_{\varepsilon}(\cdot,s))\Big| \,ds =\mathcal{J}_\varepsilon(h_{0})-
							\lim_{t\to\infty}\mathcal{J}_\varepsilon(h_{\varepsilon}(\cdot,t))\leq \mathcal{J}_\varepsilon(h_{0}).
						\end{eqnarray} This further yields the existence of a sequence $t_i\to\infty$ such that
						\[\frac{d}{dt}\mathcal{J}_\varepsilon(h_{\varepsilon}(\cdot,t_i))\to 0\ \ \ \mathrm{as}\,\,t_i\to\infty.\]
						
						By virtue of \eqref{est1}-\eqref{est4} and Lemma \ref{monotone1}, $\{h_\varepsilon(\cdot, t_i)\}_{i\in \mathbb{N}}$ is uniformly bounded and then a subsequence of $\{h_\varepsilon(\cdot, t_i)\}_{i\in \mathbb{N}}$ (which will be still denoted by $\{h_\varepsilon(\cdot, t_i)\}_{i\in \mathbb{N}}$)  converges to a positive and uniformly convex function $h_{\varepsilon,\infty}\in C^\infty(\mathbb{S}^{n-1})$ satisfying  that  \begin{equation}  {h_{\varepsilon,\infty}}(x)  |\overline{\nabla}U\big(r_{\varepsilon,\infty}(\xi)\cdot\xi\big)|^2 \det\!\big(\nabla^2h_{\varepsilon,\infty}(x)\!+\! h_{\varepsilon,\infty}(x)I\big)\!=\!\gamma_{\varepsilon}f(x)\widehat{\psi}_\varepsilon(h_{\varepsilon,\infty}), \label{equation-infty-1}\end{equation} 
						where $\xi=\alpha^*_{\varepsilon, \infty}(x)$, $r_{\varepsilon,\infty}(\xi)=(h_{\varepsilon, \infty}^2(x)+|\nabla h_{\varepsilon,\infty}(x)|^2)^{1/2}$ and   $\gamma_\varepsilon$ is a constant given by
						\[\gamma_\varepsilon=\lim_{t_i\to\infty}\frac{1}{\eta_\varepsilon(t_i)}=\frac{(n+2)T(\Omega_{\varepsilon,\infty})}{\int_{ \mathbb{S}^{n-1}}f(x)\widehat{\psi}_\varepsilon(h_{\varepsilon,\infty})dx}.\] Clearly, $h_{\varepsilon,\infty}\geq 1/C_{\varepsilon}$ on $ \mathbb{S}^{n-1}$ and then  $\Omega_{\varepsilon,\infty}\in\mathcal{K}_0$ (where $\Omega_{\varepsilon,\infty}$ is the convex body determined by $h_{\varepsilon,\infty}$). It follows from \eqref{equation-infty-1} that, for any Borel set $\omega\subset  \mathbb{S}^{n-1}$,  \begin{eqnarray}\label{varepsilon eq}
							\int_{\omega }h_{\varepsilon,\infty}(x) \,d\mu_{tor}(\Omega_{\varepsilon,\infty},x)=\gamma_\varepsilon\int_{\omega} \widehat{\psi}_\varepsilon(h_{\varepsilon,\infty})\,d\mu(x).
						\end{eqnarray} 
						
						Recall that Lemma \ref{monotone1} then implies that $T(\Omega_0)=T(\Omega_t^{\varepsilon})$, and by Lemma \ref{TVC}, we have $$T(\Omega_0)=\lim_{i\rightarrow \infty} T(\Omega_{t_i}^{\varepsilon})=T(\Omega_{\varepsilon,\infty}).$$ That is, $\Omega_{\varepsilon,\infty}\in \mathcal{K}_V$ with  
						\begin{eqnarray}\label{KV}
							\mathcal{K}_V=\Big\{\mathbb{K}\in\mathcal{K}: T(\mathbb{K})=T(\Omega_0)\Big\}.
						\end{eqnarray}  Together with \eqref{functional} and   Lemma \ref{monotone1}, one sees that $\Omega_{\varepsilon,\infty}$ solves the following optimization problem:  \[\inf\Big\{\int_{ \mathbb{S}^{n-1}}f(x)\widehat{\Psi}_\varepsilon(h_\mathbb{K}(x))\,dx: \mathbb{K}\in\mathcal{K}_V\Big\}.\]  
						
						It follows from Lemmas \ref{etaepsbound} and \ref{widthbound} that $\frac{1}{C}\le w_{\Omega_{\varepsilon,\infty}}^-\le w_{\Omega_{\varepsilon,\infty}}^+\le C$ and  $\frac{1}{C}\le\gamma_\varepsilon\le C$, respectively, where $C$ is a constant independent of $\varepsilon$.  Hence, a constant $\gamma_0>0$ and  a sequence $\varepsilon_i\to 0$ can be found so that  $\gamma_{\varepsilon_i}\to\gamma_0$, by Lemma \ref{TVC}, one can even assume that $\Omega_{\varepsilon_i,\infty}$ converges to a $\Omega_\infty\in\mathcal{K}_V$ in the Hausdorff metric. Note that, by Lemma \ref{TMC},  $ \mu_{tor}(\Omega_{\varepsilon_i,\infty})\rightarrow  \mu_{tor}(\Omega_{\infty})$ weakly. It follows from \eqref{varepsilon eq}, the fact that $\widehat{\psi}_{\varepsilon}(s)$ is bounded on $[0, C]$ and $\varepsilon\in (0, 1)$, and the dominated convergence theorem that, for each Borel set $\omega\subset \S^n$,  
						\begin{equation}\label{equ-solved-1} \int_{\omega}h_{\Omega_{\infty}} \,d\mu_{tor}(\Omega_{\infty},x)=\gamma_0\int_{\omega}\  \psi(h_{\Omega_{\infty}})\,d\mu(x). \end{equation}  Moreover, $\Omega_\infty$ satisfies
						\begin{eqnarray}\label{opti-123}
							\int_{\S^n}f\Psi(h_{\Omega_{\infty}})dx=\inf\Big\{\int_{ \mathbb{S}^{n-1}}f(x)\Psi(h_\mathbb{K})dx: \mathbb{K}\in\mathcal{K}_V\Big\}.
						\end{eqnarray}
						Clearly, equation \eqref{equ-solved-1} can be reformulated as $$  \,h_{\Omega_{\infty}} \,d\mu_{tor}(\Omega_{\infty},\cdot)=\gamma_0\psi(h_{\infty})\,d\mu,$$ where the constant $\gamma_0$, namely,\[\gamma_0=\frac{(n+2)T(\Omega_{\infty})}{\int_{ \mathbb{S}^{n-1}}f(x)\psi(h_{\Omega_{\infty}})dx}.\] 
						So $\Omega_{\infty}$ satisfies \eqref{OMPT}, and this concludes the proof of Part $(\mathrm{ii})$ in Theorem \ref{main3}. \end{proof}

					\begin{proof}[Proof of Part $(\mathrm{i})$ in Theorem \ref{main3}.] In this case, we assume that \eqref{good-condition-1}  holds, i.e., 
						\[\liminf_{s\to0^+}\frac{s^n}{\psi(s)}=\infty.\]  
					It follows from \eqref{upper11}-\eqref{etabound1} and \eqref{lower11} that \begin{eqnarray}\label{C2esti1}
						C^{-1}I\le\nabla^2h+hI\le CI. 
					\end{eqnarray} This is  a direct consequence of Lemmas \ref{lowerboundsigma_n} and \ref{boundsprincipleradii}.\\
					\indent Following the same argument as those in the proof of Part $(\mathrm{ii})$ in Theorem \ref{main3}, we can show that $h(\cdot,t)$ remains a positive, smooth and uniformly convex solution of the flow \eqref{uSSF} for all time $t>0$, due to the above priori estimates \eqref{upper11}-\eqref{etabound1}, \eqref{lower11} and \eqref{C2esti1}.  By Lemma \ref{mono},  
					\begin{eqnarray}
						\frac{d}{dt}\mathcal{J}(h(\cdot,t))\le0\ \ \ \mathrm{for\ all} \  t\in[0, \infty).
					\end{eqnarray}
					Hence, a sequence $t_i\to\infty$ can be found so that 
					\[\frac{d}{dt}\mathcal{J} (h(\cdot,t_i))\to 0\ \ \ \mathrm{as}\,\,t_i\to\infty,\] and  $h(\cdot,t_i)$ converges smoothly to a positive, smooth and uniformly convex function $h_\infty$ solving \eqref{elliptic}, where 
					\[\gamma=\displaystyle\lim_{t_i\to\infty}\frac{1}{\eta(t_i)}=\frac{(n+2)T(\Omega_{\infty})}{\int_{ \mathbb{S}^{n-1}}f(x)\psi(h_{\infty})dx},\] where $h_{\infty}$ is the support function of the convex body $\Omega_{\infty}\in \mathcal{K}_0$. In particular, $\Omega_{\infty}$ satisfies  \eqref{OMPT} with $\tau=\gamma$ and \eqref{opti-123}.
				\end{proof}

				\begin{proof}[Proof of Theorem \ref{main4}.]
					In this case, we assume $\Omega_t$ is O-symmetric. By Lemma \ref{LogC0}, the same as Proof of Part $(\mathrm{i})$ in Theorem \ref{main3}, we can get a  positive, even, smooth and uniformly convex function $h_\infty$ solving \eqref{elliptic}.

				\end{proof}

				\subsection{Proof of Theorem \ref{main1} } \label{sec-main-1-1}
				{~}
				
				Recall the definition of $\mathcal{K}_V$ by \eqref{KV}. Suppose that $\psi$ satisfy the conditions in Theorem \ref{main1}, one can obtain the existence of convex bodies in the following corollary.
				\begin{cor}\label{Part ii}
					Let $\psi$ be as in Theorem \ref{main1} and $f$ be a smooth positive function on $\S^{n-1}$, then there exist $\gamma>0$ and $\Omega\in\mathcal{K}_V$ such that $\Omega$ satisfies 
					\[\gamma \int_{\omega} \psi(h_{\Omega})\,d\mu(x)= 
					\int_{\omega}\, \,h_{\Omega} \,d\mu_{tor}(\Omega,x),\,\forall\, \textrm{Borel set}\,\,\omega\subset  \mathbb{S}^{n-1},\]and \[\int_{\S^{n-1}}f\Psi(h_\Omega)dx=\inf\Big\{\int_{ \mathbb{S}^{n-1}}f\Psi(h_\mathbb{K})dx: \mathbb{K}\in\mathcal{K}_V\Big\}.\]
				\end{cor}
				\begin{proof}
					Using standard approximation for the function $\psi$ in Theorem \ref{main1}, we can easily obtained the result by the proof of Theorem \ref{main3} directly.
				\end{proof}
				
				The proof of the following result can be found in e.g., \cite[Lemma 3.7]{CL19}.  
				\begin{lem}\label{mu}
					Let $\mu$ be a non-zero finite Borel measure on $ \mathbb{S}^{n-1}$. There is a sequence of positive smooth functions $f_j$ defined on $ \mathbb{S}^{n-1},$ such that $\,d\mu_j=f_j\,d\xi$  converges to $\mu$ weakly.
				\end{lem}

				\begin{proof}[Proof of Theorem \ref{main1}]
					Let $f_j$ be a sequence of positive and smooth functions on $ \mathbb{S}^{n-1}$ such that the measures $\mu_j$ converge to $\mu$ weakly as $j\to\infty$. By  Theorem \ref{main3},  there are $\gamma_j>0$ and $\Omega_j\in\mathcal{K}_V$
					such that
					\begin{eqnarray}\label{gdomp}
						\gamma_j \int_{\omega} \psi(h_{\Omega_j})\,d\mu_j(x)= 
						\int_{\omega}\, \,h_{\Omega_j} \,d\mu_{tor}(\Omega_j,x), 
					\end{eqnarray}  for any Borel set $\omega\subset  \mathbb{S}^{n-1}$, and \begin{eqnarray}\label{kv}
						\int_{ \mathbb{S}^{n-1}}f_j\Psi(h_{\Omega_j})dx&=&\inf\Big\{\int_{ \mathbb{S}^{n-1}}f_j\Psi(h_\mathbb{K})dx:\mathbb{K}\in\mathcal{K}_V\Big\}.
					\end{eqnarray}
					Clearly, the constants $\gamma_j$ can be calculated by 
					\begin{eqnarray}\label{constant-g-j} \gamma_j=\frac{(n+2)T(\Omega_j)}{\int_{ \mathbb{S}^{n-1}} \psi(h_{\Omega_j})d\mu_j(x)}.
					\end{eqnarray} 
					
					Let ${h_j}_{\max}=\displaystyle\max_{x\in \mathbb{S}^{n-1}}h_{{\Omega}_j}(x)$ and $\bar{x}_j\in \mathbb{S}^{n-1}$ satisfy that  ${h_j}_{\max}=h_{{\Omega}_j}(\bar{x}_j)$. Then
					\begin{eqnarray}
						h_{{\Omega}_j}(x)\ge \langle x, \bar{x}_j\rangle {h_j}_{\max} \ \ \ \ \mathrm{for\ all}\ x\in \mathbb{S}^{n-1}\ \mathrm{with}\ \langle x, \bar{x}_j\rangle >0.
					\end{eqnarray}
					Denote by $\textbf{B}^{n}_t$  the origin-symmetric Euclidean ball with radius $t>0$. Note that
					\begin{eqnarray}
						T(\textbf{B}^{n}_t)=\int_{\textbf{B}^{n}_t}|\overline{\nabla}U|^2dx,
					\end{eqnarray} and the function $t \mapsto T(\textbf{B}^{n}_t)$ is continuous and monotone increasing, then there exists a $t_0>0$ such that $T(\textbf{B}^{n}_{t_0})=T(\Omega_0)$. That is, $\textbf{B}^{n}_{t_0}\in\mathcal{K}_V$, which further yields, by \eqref{kv}, that for all $j\in \mathbb{N}$, 
					\begin{eqnarray}\label{le}
						\int_{ \mathbb{S}^{n-1}}f_j\Psi(h_{\textbf{B}^{n}_{t_0}})dx \ge\int_{\mathbb{S}^n}f_j\Psi(h_{\Omega_j})dx.
					\end{eqnarray}
					Since $f_jdx\rightarrow \,d\mu$ weakly and $\mu$ is not concentrated on any closed hemisphere, there exists a constant  $c_{\mu}>0$, such that
					\begin{eqnarray}\label{NCHj}
						\int_{ \mathbb{S}^{n-1}}\langle x,\bar{x}_j\rangle_{+}f_jdx>c_{\mu} \ \ \ \mathrm{for\ all}\ j\in \mathbb{N},
					\end{eqnarray} where $a_+=\max\{a, 0\}$ for $a\in \mathbb{R}$. To this end, assume the contrary, there exists a subsequence of $j$ (which will still be denoted by $j$) such that $x_j\rightarrow x_0\in  \mathbb{S}^{n-1}$ and 
					\begin{eqnarray}
						\int_{ \mathbb{S}^{n-1}}\langle x,x_0\rangle_{+}d\mu= \lim_{j\rightarrow \infty} \int_{ \mathbb{S}^{n-1}}\langle x,\bar{x}_j\rangle_{+}f_jdx = 0.
					\end{eqnarray}
					Consequently, one gets 
					\begin{eqnarray*}
						0=\int_{ \mathbb{S}^{n-1}}\langle x,x_0\rangle_{+}d\mu \ge\frac{\mu(\{x\in \mathbb{S}^{n-1}: \langle x,  x_0\rangle>1/N\})}{N}. 
					\end{eqnarray*} This  implies $\mu(\{x\in \mathbb{S}^{n-1}: \langle x, x_0\rangle>1/N\})=0$ for all $N>1$, which contradicts with, after taking $N\rightarrow \infty$,  the fact that $\mu$ is not concentrated on any closed hemisphere. 
					
					By  \eqref{NCHj},  for any $\zeta\in (0, 1)$, one has 
					\begin{eqnarray*}
						c_{\mu} \le\int_{ \mathbb{S}^{n-1}}\langle x,\bar{x}_j\rangle_+f_jdx
						&=&\int_{\{x\in \mathbb{S}^{n-1}:  \langle x, \bar{x}_j\rangle>\zeta\}}\langle x,\bar{x}_j\rangle_+f_jdx+\int_{\{x\in \mathbb{S}^{n-1}:   \langle x, \bar{x}_j\rangle\leq \zeta\}}\langle x,\bar{x}_j\rangle_+f_jdx
						\\&\le&\int_{\{x\in \mathbb{S}^{n-1}:  \langle x, \bar{x}_j\rangle>\zeta\}}f_jdx+\zeta \int_{ \mathbb{S}^{n-1}} f_jdx,
					\end{eqnarray*} since $0<\langle x,\bar{x}_j\rangle_+<1$.  
					Hence
					\begin{eqnarray}
						\int_{\{x\in \mathbb{S}^{n-1}:  \langle x, \bar{x}_j\rangle>\zeta\}}f_jdx \ge c_{\mu}-\zeta \int_{ \mathbb{S}^{n-1}}f_jdx.
					\end{eqnarray}
					As $\,d\mu_j= f_jdx$ converges to $\mu$ weakly, one can choose $\zeta_1 \in (0, 1)$ to be a small enough constant such that, for all $j\in \mathbb {N}$,   $$\int_{\{x\in \mathbb{S}^{n-1}:  \langle x, \bar{x}_j\rangle>\zeta_1\}}f_jdx\geq c_{\mu}-\zeta_1  \int_{ \mathbb{S}^{n-1}}f_jdx>\frac{c_{\mu}}{2}.$$ 
					Consequently, it follows that
					\begin{eqnarray}\label{ge}
						\int_{ \mathbb{S}^{n-1}}\Psi(h_{\Omega_j})f_jdx&\ge& \int_{\{x\in \mathbb{S}^{n-1}:  \langle x, \bar{x}_j\rangle>\zeta_1\}} \Psi({h_j}_{\max}  \langle x, \bar{x}_j\rangle)f_jdx
						\nonumber\\&\ge& \Psi(\zeta_1 {h_j}_{\max}) \int_{\{x\in \mathbb{S}^{n-1}:  \langle x, \bar{x}_j\rangle>\zeta_1\}}f_jdx \nonumber\\&\ge&\frac{c_{\mu}}{2} \cdot\Psi(\zeta_1 {h_j}_{\max}).
					\end{eqnarray}
					
					As $\,d\mu_j= f_jdx$ converges to $\mu$ weakly, \eqref{le}  implies that for all $j\in \mathbb{N}$, 
					\begin{eqnarray*} \int_{ \mathbb{S}^{n-1}}f_j\Psi(h_{\Omega_j})dx\leq \sup_{j\in \mathbb{N}}
						\int_{ \mathbb{S}^{n-1}}f_j\Psi(h_{\textbf{B}^{n}_{t_0}})dx <\infty,
					\end{eqnarray*} this together with \eqref{ge} and the fact that  $\Psi(s)\to\infty$ as $s\to\infty$, one gets  that ${h_j}_{\max}\leq C$, where we use $C>0$ to denote a constant
					which depends only on $\mu$ and $\psi$, but it may change from line to line.\\
					Since $\Omega_j\in\mathcal{K}_V$, using the same argument as in Lemma \ref{widthbound}, we obtain
					\begin{equation}\label{width2}
						1/C \le w_{\Omega_j}^-\le w_{\Omega_j}^+\le C, \quad \forall j.
					\end{equation} 
					\indent By \eqref{width2} and the Blaschke selection theorem, we obtain the existence of a subsequence $\{\Omega_j\}_{j\in \mathbb{N}}$ (which will still be denoted by $\{\Omega_j\}_{j\in \mathbb{N}}$) and a convex body $\Omega\in\mathcal{K}_V$ such that $\Omega_j\rightarrow \Omega.$ 
					It follows from \eqref{gdomp}, \eqref{constant-g-j}, Lemma \ref{TMC} and the dominated convergence theorem we have $\gamma_j\to\gamma>0$ and \begin{eqnarray*}
						\gamma\int_{\omega} \psi(h_{\Omega})\,d\mu(x)= 
						\int_{\omega}\, \,h_{\Omega} \,d\mu_{tor}(\Omega,x),  
					\end{eqnarray*}  for each Borel set $\omega\subseteq  \mathbb{S}^{n-1}$. This  concludes that $\Omega$ satisfies \eqref{OMPT} as desired.
				\end{proof}

				\section{Appendix: $C^2$-estimates} \label{section-5}
				This section is devoted to the second order derivative estimates. Since the equation studied in this paper is not quite the same as the equation in the appendix  of \cite{LSYY21}. In fact, the Lemma A.1. in \cite{LSYY21}, $\Phi$ is smooth function independent of $\Omega_t$, but in this paper the $\overline{\nabla} U$ depends on $\Omega_t$. We can reduce the current case to that one in \cite{LSYY21}. 
				
				\begin{lem}\label{lowerboundsigma_n} Let $T>0$ be a constant  and $h(\cdot, t)$ be a positive, smooth and uniformly convex solution to
					\begin{eqnarray}\label{geneflow}
						\frac{\partial h}{\partial t}(x,t)=-f(x)\psi(h) |\overline{\nabla} U(\overline{\nabla}h,t)|^{-2}\big(\det(\nabla^2 h+hI)\big)^{-1}+h\eta(t).
					\end{eqnarray}  If there is a  positive constant $C_0$ such that,  for all $x\in \mathbb{S}^{n-1}$ and all $t\in [0,T)$,  
					\begin{eqnarray}
						\label{cond1}1/C_0\le h(x,t)\le C_0, \\
						\label{cond2}|\nabla h|(x,t)\le C_0, \\
						\label{cond3}0\le|\eta(t)|\le C_0,
					\end{eqnarray} then the following statement holds: for all $(x, t)\in\mathbb{S}^{n-1}\times [0, T)$, 
					\begin{eqnarray}\label{lower bound sigma}
						\det(\nabla^2h+hI)\ge1/C,
					\end{eqnarray}
					where $C>0$ is a constant depending on $\Omega_0$ (the initial hypersurface), $C_0$, $|\psi|_{L^\infty(D)}$, $|\psi^{'}|_{L^\infty(D)}$with $D=[1/C_0,C_0]$, $\min_{t\in[0,T)}|\overline{\nabla}U|_{L^\infty(\partial\Omega_t)}$ and $\max_{t\in[0,T)}|\overline{\nabla}^k U|_{L^\infty(\partial\Omega_t)}, k=1,2$.
				\end{lem}
				\begin{proof} Since $\overline{\nabla}h(x, t)=\nabla h(x, t)+h(x, t)x$, we may denote
					$$ f(x)\psi(h) |\overline{\nabla} U(\overline{\nabla}h,t)|^{-2}=f(x)\psi(h) |\overline{\nabla} U(\nabla h(x, t)+h(x, t)x,t)|^{-2}=:\Phi (x, h, \nabla h),$$
					which is just the same as $\Phi$ in Lemma A. 1. in \cite{LSYY21}. So we can get the result.
					
				\end{proof}

					The following lemme provides the boundedness (from both above and below) of the principal curvature radii of $\mathcal{M}_t$. Its proof also comes from \cite{LSYY21}.
					\begin{lem}\label{boundsprincipleradii} Let $T$ be as stated in Lemma \ref{lowerboundsigma_n}. Let  $h(x, t)$ be a positive, smooth and uniformly convex solution to \eqref{geneflow} and satisfy conditions  \eqref{cond1}-\eqref{cond3}. Then, for all $t\in[0,T)$ and $x\in \mathbb{S}^{n-1}$, one has
						\begin{eqnarray}\label{upperlower}
							(\nabla^2h+hI)(x,t)\le C I,\,\,\forall (x, t)\in \mathbb{S}^{n-1}\times [0, T),
						\end{eqnarray}
						where $C$ is a positive constant depending on $C_0$, $h(\cdot,0)$, $|\psi|_{L^\infty(D)}$, $|\psi^{'}|_{L^\infty(D)}$with $D=[1/C_0,C_0]$, $\max_{t\in[0,T)}|\overline{\nabla}^k U|_{L^\infty(\partial\Omega_t)}, k=1,2,3$ and $\min_{t\in[0,T)}|\overline{\nabla} U|_{L^\infty(\partial\Omega_t)}$.
					\end{lem}

						\newpage


					\end{document}